\documentclass{amsart}

\usepackage[mathcal]{euscript}
\usepackage{mathrsfs}
\usepackage{amsfonts,amssymb,amsbsy,amsmath}
\usepackage{graphicx,color}
\usepackage{epsfig}
\usepackage[all,cmtip]{xy}
\usepackage{verbatim}
\usepackage{multirow}
\usepackage{lscape}
\usepackage{pdfpages}

\newcommand{\R}{\mathbb{R}}
\newcommand{\Q}{\mathbb{Q}}
\newcommand{\C}{\mathbb{C}}
\newcommand{\Z}{\mathbb{Z}}

\newcommand{\disp}{\displaystyle}
\newcommand{\ba}{\begin{array}}
\newcommand{\ea}{\end{array}}

\newcommand{\be}{\begin{equation}}
\newcommand{\ee}{\end{equation}}

\theoremstyle{plain}
\newtheorem{theorem}{Theorem}
\newtheorem{lemma}[theorem]{Lemma}
\newtheorem{prop}[theorem]{Proposition}
\newtheorem{coro}[theorem]{Corollary}

\theoremstyle{remark}
\newtheorem{remark}{Remark}

\theoremstyle{definition}

\usepackage[utf8]{inputenc}

\begin{document}


\title{Real tight contact structures on lens spaces and surface singularities}

\author{S\.{\i}nem Onaran}
\address{Hacettepe \"{U}n\.{ı}vers\.{ı}tes\.{ı}, Department of Mathematics, TR-06800 Beytepe-Ankara, Turkey}
\email{sonaran@hacettepe.edu.tr}

\author{Fer\.{\i}t \"{O}zt\"{u}rk}
\address{Bo\u{g}az\.{ı}\c{c}\.{ı} \"{U}n\.{ı}vers\.{ı}tes\.{ı}, Department of Mathematics, TR-34342
  Bebek, \.Istanbul, Turkey}
\email{ferit.ozturk@boun.edu.tr}

\subjclass[2020]{Primary 57K33, 32S25; Secondary 57M50}

\begin{abstract}
We classify the real tight contact structures on solid tori up to equivariant contact isotopy and apply the results to the classification of real tight structures on  $S^3$ and real lens spaces $L(p,\pm 1)$. We prove that there is a unique real tight $S^3$ and $\R P^3$. We show there is at most one real tight  $L(p,\pm 1)$ with respect to one of its two possible real structures. With respect to the other we give lower and upper bounds for the count. To establish lower bounds we explicitly construct real tight manifolds through equivariant contact surgery, real open book decompositions and isolated real algebraic surface singularities.
As a by-product we observe the existence of an invariant torus in an $L(p,p-1)$
which cannot be made convex equivariantly.

\end{abstract}

\date{\today}

\maketitle

\section{Introduction}
A real structure on an oriented 3- or 4-manifold is an orientation preserving involution.
A real contact 3-manifold is a closed real 3-manifold where the contact structure is anti-symmetric with respect to the real structure. 
They naturally arise as the link manifolds of complex algebraic isolated surface singularities given by zero loci of polynomials with real coefficients (in this note we term those singularities as real algebraic surface singularities);
there the real structure is the complex conjugation and the contact structure is the unique tight one determined by complex tangencies. 
Without reference to the real structure, such a singularity is usually called a Milnor filling of the given contact 3-manifold. So in the real setting we may take the liberty to call the filling a real Milnor filling.
More generally real contact 3-manifolds appear as the boundary of real symplectic 4-manifolds in the sense of \cite{vi}.

It is known that the 3-manifolds that admit a real structure is scarce.  ``Most manifolds have little symmetry''; more precisely there is a natural notion of density in the set of all closed orientable 3-manifolds defined via the first cohomology and this density vanishes for the subset of those manifolds that admit a non-trivial involution  \cite{pu}. Nevertheless once a 3-manifold admits a real structure $c$ then it admits a $c$-real contact structure \cite{ose}.

Having noted these, a possible path of investigation is real Milnor fillability, 
i.e. the fillability of real tight contact 3-manifolds by 
isolated real algebraic surface singularities, and more generally the fillability by real symplectic 4-manifolds. Related to the latter, it has been observed that  
there are real open books (hence real contact structures) on real 3-manifolds  which cannot be filled by a real Lefschetz fibration, although it is filled by non-real Lefschetz fibrations \cite{os3}.

We started this work with an interest in realizing real tight lens spaces as the link manifolds of real algebraic surface singularities, in cases where there is no contact topological obstruction.  Recall that up to contactomorphism there is at most one tight contact structure on a closed oriented 3-manifold which is Milnor fillable \cite{cnp}. Keeping that in mind, we first attempt a classification of real tight structures over lens spaces up to equivariant contact isotopy. In that regard the present work can be considered as a continuation of  \cite{os} in which we introduced real contact 3-manifolds and showed that there is a unique real tight 3-ball, which was accomplished through a partial classification of real tight solid tori relative to the invariant convex boundary.

Here in order to proceed towards the classification on lens spaces, we first improve the classification on real tight solid tori. We also expect this classification to be useful in studying the classification of real Legendrian knots and hope to return this in a future work.
On a solid torus there are four real structures up to isotopy 
through real structures \cite{ha}. We denote them and their restrictions on the boundary torus by $c_j, j=1,2,3,4$ (see Section~\ref{ilk} for an explicit list). 
Here $c_1$ is the hyperelliptic involution.
It turns out that the $c_2$-, $c_3$- and $c_4$-real tight solid tori with convex boundary appear only as the standard real tight neighborhoods of Legendrian and real knots (Theorem~\ref{anadolu}; see Section~\ref{komsuluk} for the discussion on standard neighborhoods).

Using these results we deduce that there is a unique real tight structure  on $S^3$ (Theorem~\ref{S3}) and on $\R P^3$  (Theorem~\ref{rp3}). These occur as the link manifolds of a regular point and an $A_1$ singularity of a real algebraic surface respectively (Section~\ref{s3rp3}).

On lens spaces $L(p,q)$ for $p>2$, $q=1,p-1$ up to equivariant isotopy there are exactly two real structures with nonempty real part \cite{hr}. On a genus-1 Heegaard decomposition these act
on the boundaries of the solid tori either as  $(c_1, c_1)$ or as $(c_2, c_3)$, called  
type~$A$ and type~$B$ respectively. Moreover each is equivariantly isotopic to
a real Heegaard decomposition of genus 1;  i.e. the real Heegaard genus equals 1 (type~$C$ or $C'$), See Section~\ref{AveB} for details. There we prove

\begin{theorem}
(Theorem~\ref{hudutB} in Section~\ref{AveB}) There is a $B$-real tight structure on $L(p,q)$ if and only if $q=1$ or $q=p-1$. Up to equivariant contact isotopy, there is a unique $B$-real tight structure on $L(p,p-1)$ and on $L(p,1)$, $p$ even. They are real Milnor fillable. There is at most one $B$-real tight structure on $L(p,1)$, $p$ odd. 
\end{theorem}

We explicitly construct the claimed structures via isolated real algebraic surface singularities and equivariant contact diagrams  (see \cite{ose} for the foundations of equivariant contact diagrams).

As a curious by-product of this result, we observe that there is 
an invariant Heegaard torus in an $L(p,p-1)$
which cannot be made convex equivariantly, although of course it can always be made convex (Theorem~\ref{l*Cp-1}). Equivalently the unique $B$-real tight $L(p, p-1)$ does not admit a real contact Heegaard decomposition of genus~1; i.e. its real contact Heegaard genus is greater than 1.
The proof is through observing the mismatch between the Thurston-Bennequin invariants of the real Legendrian knots in  
a $B$-real tight  $L(p, q)$ and a hypothetical $C$-real tight 
$L(p, q)$, which are known to be equivariantly isotopic. (See Section~\ref{CandC'} for the discussion on type~$C,C'$ and the computation of the invariants.)

As for the classification in type~$A$, one needs a classification for $c_1$-real tight solid tori and thick tori. The latter is required in peeling off basic $c_1$-slices from solid tori \cite{ho}. We prove that such a peeling-off is impossible by showing that interestingly there is no $c_1$-real tight basic slice (Proposition~\ref{nobasic}) 
although there are  $c_1$-real tight double slices (Proposition~\ref{vardouble}). 
Nevertheless we give a partial classification for $c_1$-real tight solid tori and upper bounds for the count (see Section~\ref{c1soy}).

Using these results we demonstrate 
upper and lower bounds for the count of  $A$-real tight lens spaces $L(p,q)$ for $q=1,p-1$ (Theorem~\ref{hudutA}).
The lower bounds are established by explicit construction through equivariant contact diagrams, real open book decompositions (Section~\ref{realob}) and real algebraic singularities (Section~\ref{q=p-1}). 
The lower bound in each case agrees with the definite count for the number of tight structures of Ko~Honda \cite{ho}. 
We also point out the  cases of real tight lens spaces which are not real Milnor fillable, the obstruction being contact topological. Thus we observe that fillability by isolated real algebraic singularities is neither related with the type of the real structure, nor with the real Heegaard genus.

\section{Preliminaries}
\label{ilk}

The background for the 3-dimensional contact topology  and the techniques we use in the sequel can be found in several standard references such as  \cite{et}, \cite{ge}, \cite{gi}.

In order to fix the notation and refer properly, we recall the classification
of real structures on a solid torus up to isotopy  through real structures \cite{ha}.
We identify the oriented boundary a solid torus $S^1\times D^2$ with $\R^2_{(x,y)}/ \Z^2$ where $x$ and $y$ 
directions correspond to the direction of the meridian $(1,0)$ and of longitude $(0,1)$ respectively.
We fix the coordinates of 
$S^1\times D^2$  as  $(y, (u, v))$ where $(u,v)$ is the rectangular coordinates over $D^2$.
Then the four possible real structures on the solid torus are given as  follows:

\noindent (1) $c_1:(y,(u,v))\mapsto (-y,(-u,v))$ -- hyperelliptic; \raisebox{-0.15cm}{\mbox{\includegraphics[height=.5cm]{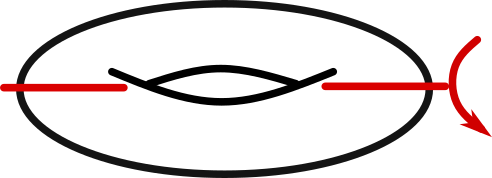}}} \\
(2) $c_2:(y,(u,v))\mapsto (y,(-u,-v))$; \raisebox{-0.15cm}{\mbox{\includegraphics[height=.5cm]{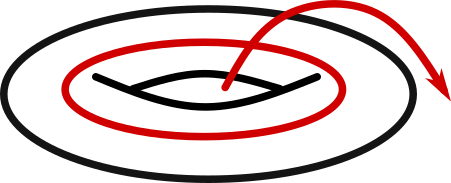}}}\\
(3) $c_3:(y,(u,v))\mapsto (y+1/2,(u,v))$; \raisebox{-0.25cm}{\mbox{\includegraphics[height=.5cm]{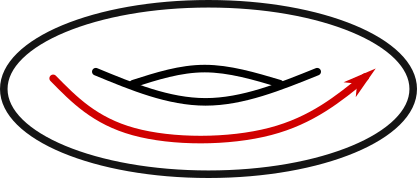}}}\\
(4) $c_4=c_2\circ c_3:(y,(u,v))\mapsto (y+1/2,(-u,-v))$. \raisebox{-0.1cm}{\mbox{\includegraphics[height=.5cm]{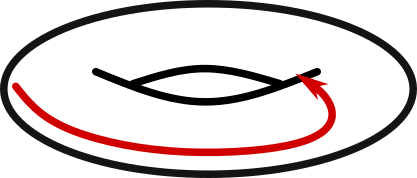}}} \\

All real structures and real parts in the sequel will be sketched in red, whenever color is applicable.

For $a\in \Z-\{0\}$ we set 
$$\eta_{a}\doteq \ker(\cos(2\pi a y) du + \sin(2\pi a y)  dv),$$ 
and 
$$ \eta_0\doteq \ker(du+ \sin(2\pi a y)  dv).$$ 
Note that the contact solid torus $(S^1\times  D^2, \eta_a)$ is $c_1$- and  $c_2$-real tight; it is  $c_3$-real  tight for $a$ odd and  $c_4$-real tight for $a$ even. 
We set $$ \zeta\doteq \ker(dy+ u dv).$$   
Note that $(S^1\times  D^2, \eta_a)$ is $c_1$-real tight.

\subsection{Real tight neighborhood theorems for knots}
\label{komsuluk}
In \cite{os}, we have presented contact neighborhood theorems for real submanifolds 
and invariant contact submanifolds up to equivariant contactomorphism. Also we have given the classification of sufficiently small equivariant neighborhoods of real knots in a real 3-manifold up to equivariant contact isotopy.

Here, we resolve shortly the remaining cases for the classification of real tight neighborhoods of equivariant transverse and Legendrian knots up to equivariant  contactomorphism and up to isotopy through real tight structures.

\begin{theorem} (Real Contact Neighborhood Theorem for Equivariant Transverse  Knots)
An equivariant transverse knot $K$ in a real contact 3-manifold $(M,\alpha,c)$
has a neighborhood equivariantly contactomorphic to a standard one, namely to the real tight solid torus
$(S^1\times D^2,\zeta,c_1)$.
\end{theorem}

\noindent {\it Proof:} 
By \cite[Theorem~2.6]{os},  it is enough to classify the real conformal symplectic 2-bundles over $S^1$ up to equivariant conformal bundle isomorphism. 

An involution on the unit circle  $S^1$ in $\C$ is isotopic through involutions either
to the reflection with respect to the $x$-axis or to the antipodal map. We denote
respectively by $c_1$ and $c_3$ the corresponding actions of the real structure $c$ on $K$. 
Note that $c_3$ is orientation preserving on $K$ while $c_1$ is reversing.

We will argue that there is no real conformal symplectic 2-bundle over $S^1$ with its real
structure inducing $c_3$ on the zero section. Let there be such a bundle on $S^1$, with real
structure still denoted by $c_3$ (with a slight abuse of notation). The involution $c_3$ is orientation preserving
on the fibers, sending each fiber to the antipodal one. Recall that the symplectic structure $\omega$ on
the fibers is anti-symmetric, i.e. $c_3^*\omega=-\omega$. Thus over the  semi-circle, the 2-form $\omega$ is isotoped 
through symplectic (area) forms to $-\omega$. This is impossible.

Now we want to classify the real conformal symplectic 2-bundles over $S^1$ with the real
structure inducing $c_1$ on the zero section.
Consider two copies $N_1$ and $N_2$ of the canonically oriented trivial bundle 
$[0,\pi]\times \R^2$ with the standard symplectic form on each fiber.
Then  any real conformal symplectic 2-bundle $N$ can be built as a {\it real mapping torus} as follows: 
$$ N^{f,g}= \frac{N_1 \amalg N_2}{(\pi,(u,v))\sim(0,f(u,v)),(0,(u,v))\sim(\pi,g(u,v))}$$
where $f,g\in \mbox{Sp}(2)$ and are equivariant with respect to the real structure $c_1$ on $N$ given by 
$$c_1:N_i^{f,g}\rightarrow N_j^{f,g},\;\; (y,(u,v))\mapsto (1-y,(u,-v)), \;\;\; i,j\in\{1,2\}, i\neq j.$$ 
Here $N_i^{f,g}$ ($i=1,2$)  denote the projection of $N_i$ to $N^{f,g}$. Then $f$ and $g$ can be either $+$id and $-$id. 
Note that with the choices of orientation above, $N^{f,g}$ is canonically oriented and $c_1$ is orientation preserving.
Moreover the four possible bundles $N^{++},N^{--},N^{+-},N^{-+},$ are equivariantly conformal isomorphic. Explicitly $N^{++}$ and $N^{--}$ are equivariantly isomorphic via the diffeomorphism 
\begin{eqnarray*}
N^{++}_1\rightarrow N^{--}_1, 
(y,(u,v))\mapsto (y,u\sin y+v\cos y,-u\cos y+v\sin y) \\
N^{++}_2\rightarrow N^{--}_2, 
(y,(u,v))\mapsto (y,u\sin y+v\cos y,-u\cos y+v\sin y).
\end{eqnarray*}
$N^{++}$ and $N^{+-}$ are equivariantly isomorphic
via the diffeomorphism 
\begin{eqnarray*}
N^{++}_1\rightarrow N^{+-}_1, 
(y,(u,v))\mapsto (y,u\cos\frac{y}{2}+v\sin\frac{y}{2},-u\sin\frac{y}{2}+v\cos\frac{y}{2}) \\
N^{++}_2\rightarrow N^{+-}_2, 
(y,(u,v))\mapsto (y,u\sin\frac{y}{2}-v\cos\frac{y}{2},u\cos\frac{y}{2}+v\sin\frac{y}{2}).
\end{eqnarray*}
(Similarly for the other cases.) \hfill $\Box$ \\

As for an equivariant Legendrian knot $K$, as in \cite[Corollary~2.5]{os}, we observe that the proof of  \cite[Theorem~2.2]{os} shows that the isotopy classification of the real contact structures near $K$ is real topological. Therefore $tw(K)$ with respect to a longitude and the action of the real structure in a neighborhood of $K$ determine the standard tubular neighborhood of $K$ up to equivariant contact isotopy. 
Thus we obtain the following classification; the case (2) has already been presented as \cite[Corollary~2.5]{os} except for $tw=0$. 

\begin{theorem} (Real Contact Neighborhood Theorem for Equivariant Legendrian  Knots)
\label{eqLeg}
Let $K$ be an equivariant Legendrian  knot in a real tight $(M,\xi,c)$.
Let $U$ be a sufficiently small real contact neighborhood of $K$ such that $\partial U$ is an equivariant  convex torus.
Then $U$ is equivariantly contact isotopic through real tight structures to one of the following:
\begin{enumerate}
\item $c|_K=c_1$: $U\sim(S^1\times D^2,\pm\eta_{tw(K)},c_1)$. Here we assume that  the real arcs are oriented so that the real points on the boundary are labeled. Then the sign $\pm$ is determined by the choice of sign decoration on the regions of $\partial U - \Gamma$.
\item $c|_K=id$: $K$ is real and $U\sim(S^1\times D^2,\eta_{tw(K)},c_2)$.
\item $c|_K=c_3$: if $tw(K)$ is odd then $U\sim (S^1\times D^2,\eta_{tw(K)},c_3)$;
if $tw(K)$ is even then $U\sim (S^1\times D^2,\eta_{tw(K)},c_4)$.
\end{enumerate}
\end{theorem}

\begin{remark}
\label{artieksic1}
Note that in case (1), the two contact structures $\pm\eta_{tw(K)}$ are not equivariantly contact isotopic simply because the diffeomorphism $c_1$ that sends one to the other is not equivariantly isotopic to the identity: the labeled real points on the boundary are to be fixed.  (A discrepancy between real and non-real cases... These contact structures do not 
satisfy the Giroux's Flexibility in the real setting; e.g. the hypothesis of \cite[Corollary~3.5]{os} is not satisfied.)
\end{remark}

\section{Real tight solid tori}

We gave a partial classification  of $c_3$-real tight solid tori with $\#\Gamma=2$ in \cite{os}.
Here we give a complete classification for all $c_j$-real tight solid tori $j=2,3,4$ with $\#\Gamma=2$ and a partial classification for $c_1$-real tight solid tori.
First, let us set up the required terminology.

Consider two points $a,b\in\Q$ located at the boundary of  the hyperbolic unit disk with the Farey tessellation.
We consider the longest counterclockwise steps in the Farey tessellation.
Starting from $a$, if one can reach $b$ after $d$ such steps, we say that the distance between 
$a$ and $b$, denoted $dist(a,b)$, is $d$.


Below, we are going
to split contact manifolds into building blocks. These are either basic slices $T^2\times I$ as defined in \cite{ho} supplied with $c$-real tight contact structures or $c$-invariant genuine double slices. Contact topologically, 
a {\em genuine double slice} is nothing but two basic slices glued together, which cannot be split into two  $c$-invariant basic slices though.
We will show in the sequel that depending on $j$, $c_j$-invariant basic slices or genuine double slices may or may not exist.
We denote by $c$-$T(a,b)$ a $c$-real tight, $c$-invariant thick torus with convex boundary with inner slope $a$ and outer slope $b$.
Note that $dist(a,b)$ on a basic slice and on a double slice is $1$ and $2$ respectively.

On a real tight torus, let us denote by  $\mu$ and $\lambda$ the meridian and the longitude and 
by $\tau_{\mu}$ and $\tau_{\lambda}$ the Dehn twist along each respectively. 
Observe that $\tau^2_{\mu}$ and $\tau_{\lambda}$ (resp. $\tau_{\mu}$ and $\tau^2_{\lambda}$) can be realized in a $c_3$-invariant way (resp. $c_2$-invariant). Both $\tau_{\mu}$ and $\tau_{\lambda}$ can be made $c_1$-invariant.

\subsection{Stripping-off $c_3$-real tight solid tori}
\label{c3soy}

We want to break a $c_3$-real tight solid torus with  boundary slope $s$ into equivariant slices.
We proceed as in the proof of \cite[Theorem~11]{os}. 
After applying a sequence of $c_3$-invariant Dehn twists $\tau^2_{\mu}$, we can assume that  $s\leq -1$ or $s\geq 1$. 
An equivariant   Legendrian isotopic copy  $L$ of the core has a standard equivariant tight neighborhood
with equivariant and convex boundary. The boundary slope is $-\frac{1}{2k+1}$ (by Theorem~\ref{eqLeg}
or \cite[Corollary~15]{os}).
We decompose the solid torus into two equivariant pieces: $\nu(L)$ and  $T(-\frac{1}{2k+1},s)$.
The latter is  minimally twisting by Proposition~4.16 and the proof of Proposition~4.15 in \cite{ho}.
We  factor $T(-\frac{1}{2k+1},s)$ into equivariant slices. If the final slice is
a basic slice of the form  $T(-\frac{1}{2k+1},-\frac{1}{2k})$ or $T(-1,-2)$,
we have shown in \cite{os} that this  leads to an overtwisted structure.

Otherwise we have a final slice in one of the following forms:\\
(i) Either  $T(-1,t)$ where $t\in \Q\cap[-3,-1)$ and $dist (-1,t)=1$ or $2$ (single or double slice respectively); \\
(ii) Or $T(-\frac{1}{2k+1},-\frac{1}{2k-1})$.

Now via applying a chain of $c_3$-equivariant homeomorphisms, second case above can be transformed to
a thick torus $T(-1,p)$ with $p\in\Z^+$ and the double slices in the first case above can be transformed to such a $T(-1,p)$ or a $T(-1,-\frac{p}{q})$ with $p=2k+1$, $q=k$, $k\in\Z^+$.
In particular there are two possible $c_3$-real tight single slices: $T(-1,-2)$ or $T(-1,\infty)$, and the former was already proven impossible in  \cite[Proposition~4.11]{os}.

\subsection{Stripping-off $c_2$-real tight solid tori}
\label{c2soy}

With respect to the real structure $c_2$, $\tau_{\mu}$ is 
equivariant while $\tau_{\lambda}$ is not. By means of a $c_2$-equivariant $\tau_{\mu}$ a boundary slope can always be made less than or equal to $-1$.

Suppose $(S^1 \times D^2,\xi,c_2)$ is a real tight solid torus with convex boundary and with
$\#\Gamma = 2$.
Let $s=-p/q$ $(p,q\in \Z, p>q\geq 0)$.


The core $L$ of the solid torus is real, hence Legendrian.
Let $k=\mbox{tw}(L)$.
By \cite[Corollary~2.5]{os}, a sufficiently small neighborhood $U$ of $L$ is equivariantly contact isotopic
to $U_k$. Excise $U$ to get an equivariant real thick torus $(T^2\times I,\xi,c_2)$ with 
convex boundary components $T_0=T^2\times\{0\}$ and $T_1=T^2\times\{1\}$ with $\#\Gamma=2$ on each and with slopes
$s_0=-1/k$ and $s_1=-p/q$ respectively. We denote such a thick torus  by $T(-1/k,-p/q)$ as before.

We proceed very much similarly as in the above paragraph; we decompose $T(-1/k,-p/q)$ 
into equivariant basic and double slices. 
Namely, suppose that the equivariant characteristic foliation (Legendrian rulings) 
has slope 0. 
A  meridional annulus has Legendrian boundary with negative twisting with respect to the annulus and thus  
can be perturbed relative to the boundary to obtain a convex annulus $A$ (\cite{ho} or \cite{ka}). 
The inner boundary component of $A$ has tw $-1$ while outer has $-p$. If $p>2$,
there are at least two bypasses on $A$ for $T_1$. We can perform one
bypass after the other on $T_1$ to obtain a $c_2$-invariant convex torus $T'_1$.
Here we note that after the first bypass, the shifted copy $\gamma$ of the meridian (the old meridional arc union the inner part of the bypass halfdisk) is still a leaf of the characteristic foliation. On the disk bounded by $\gamma$ there is still a bypass which can be made symmetric with the first.
The thick torus
bounded by $T_1$ and $T_1'$ is a $c_2$-invariant double slice. Moreover, the distance 
from $s_1$ to the slope $s_1'$ on $T_1'$ is 2. 

This discussion shows that after  a number of steps, we obtain a collection of $c_2$-invariant 
double slices and possibly a single invariant basic slice. Furthermore, the last slice
is either $T(-\frac{1}{k},-\frac{1}{k-1})$ or $T(-\frac{1}{k},-\frac{1}{k-2})$ or
$T(-1,-\frac{p}{q})$, where the distance from $-\frac{p}{q}$ to $-1$ is 1 or 2 and 
$-\infty\leq -\frac{p}{q}<-1$. Whatever the last slice is, we can find a
$c_2$-equivariant diffeomorphism to change $s_0$ to $-1$ by $\tau_{\mu}$ and then modify $s_1$ by $\tau^2_{\lambda}$ to obtain a real tight slice in one of the forms $T(-1,-2)$ or $T(-1,-\infty)$ or $T(-1,-\frac{p}{q})$ or $T(-1,p)$ where $p\in \Z^+$ and $-3\leq-\frac{p}{q}<-2$, 
$dist(-1,-\frac{p}{q})=2$.
In particular, $T(-\frac{1}{k},-\frac{1}{k-1})$ is equivariantly contactomorphic to $T(-1,-\infty)$ and $T(-\frac{1}{k},-\frac{1}{k-2})$ to $T(-1,+1)$.

We claim that a $c_2$-real, tight, minimally twisting $T(-1,-2)$ does not exist. In fact, 
consider a convex equivariant  meridional annulus.
The inner boundary component of the annulus has $tw=-1$ while outer has $-2$.
The dividing set on the annulus must be $c_2$-symmetric and, because of the minimally twisting property, it must have at least one component connecting the inner boundary to the outer. However such a dividing set  is impossible.

\subsection{Classification of $c_2$-, $c_3$- and $c_4$-real tight solid tori}

First we will note the obvious correspondence between $c_2$-real tight thick tori
and $c_3$-real tight thick tori.

Let $T(-1,b)$ be a $c_3$-real tight thick torus. The map $\phi: T(a,b)\rightarrow T^2\times I$
given by the matrix $\begin{pmatrix} 0 & 1 \\ -1 & 0 \end{pmatrix}$ supplies the thick torus with a
$c_2$-real tight structure. With that structure, it is a $c_2$-real $T(1,-1/b)$. Via the $c_2$-equivariant
diffeomorphisms $\begin{pmatrix}  1 & -2 \\ 0 & 1 \end{pmatrix}$, we can turn that thick torus into a $c_2$-real $T(-1,-1/(b+2))$.
Meanwhile via the $c_2$-equivariant
diffeomorphisms $\begin{pmatrix}  1 & 0 \\ -2 & 1 \end{pmatrix}$, we get a $c_2$-real $T(-1,-(2b+1)/b)$.
Note that the above discussion is still valid if we replace the roles of $c_3$ and $c_2$.
Thus we have proven

\begin{lemma}
\begin{enumerate}
\item There is a $c_3$-real (respectively $c_2$-real) tight thick torus $T(-1,k)$, $(k\in \Z^+)$ if and only if 
there is a $c_2$-real  (respectively $c_3$-real)  tight thick torus $T(-1,-(2k+1)/k)$.
\item There is a $c_3$-real (respectively $c_2$-real) tight thick torus $T(-1,k)$, $(k\in \Z^-, k<2)$ if and only if 
there is a $c_2$-real  (respectively $c_3$-real)  tight thick torus $T(-1,-1/(k+2))$.
\end{enumerate}
\label{oha}
\end{lemma}

Hence we obtain 
\begin{theorem}
There is no $c_2$- and $c_3$-real tight double slice.
\label{vay}
\end{theorem}
\noindent {\it Proof:}
There are two possible sorts of $c_3$-real tight double slices. For  $k\in \Z^+$, a  $c_3$-$T(-1,k)$ to exist, there must exist a $c_2-T(-1,-1/(k+2))$ by Lemma~\ref{oha}.(2). Stripping off this latter
thick torus into basic slices, at the end of the first step we obtain a tight $T(-1,0)$ which does not exist.
Therefore there is no  $c_3$-$T(-1,k), (k\in \Z^+)$.

Similarly a  $c_3$-$T(-1,-(2k+1)/k), (k\in \Z^+)$ exists if and only if a  $c_2-T(-1,k)$  exists (Lemma~\ref{oha}.(1))
if and only if a  $c_3$-$T(-1,-1/k)$ exists  (Lemma~\ref{oha}.(2)). Then the proof follows as above.
 \hfill $\Box$

The discussion above helps us prove that $c_2$-, $c_3$-  and $c_4$-real tight solid tori are scarce in the following sense:

\begin{theorem}
\label{anadolu}
The classification of $c_2$-, $c_3$-  and $c_4$-real tight solid tori with convex boundary with $\#\Gamma=2$:

The standard neighbourhoods of Legendrian knots with convex boundary on which the real structure is $c_3$ (respectively $c_4$) are the only $c_3$-real (respectively $c_4$-real) tight solid tori up to equivariant contact isotopy.

Similarly the standard neighbourhoods of real knots with convex boundary are  the only $c_2$-real tight solid tori up to equivariant contact isotopy.
\end{theorem}

\noindent {\it Proof:} 
Given a $c_2$-, $c_3$ or $c_4$- real tight solid torus, we take a close, Legendrian copy of the core, excise
its standard equivariant neighbourhood and start peeling-off equivariant  slices  from the
remaining thick torus. The discussion in Sections~\ref{c3soy} and \ref{c2soy} has already eliminated
all the possibilities except a $c_2$-real single slice $T(-1,-\infty)$,  the existence of which is equivalent to
the existence of a $c_3$-real single slice $T(-1,-2)$ (Lemma~\ref{oha}.(2)), which does not exist  by \cite[Proposition~4.11]{os}

Similarly instead of peeling-off equivariant slices in case of $c_4$, we turn the initial $c_4$-real tight thick torus 
into a $c_3$-real tight thick torus via a diffeomorphism 
which acts as $\tau_{\mu}$ on each $T^2\times \{*\}$. Since a  $c_3$-real tight basic or double slice does not exist (Theorem~\ref{vay}), the proof for the case $c_4$ follows.
\hfill $\Box$

In particular, we have
\begin{coro} \label{egimler}
A $c_3$-real tight solid torus can only have slopes $1/k$, $k$ an odd integer (Theorem~\ref{eqLeg} and \cite[Corollary~4.4]{os}). A $c_4$-real tight solid torus can only have slopes $1/k$, $k$ an even  integer (Theorem~\ref{eqLeg}).
A $c_2$-real tight solid torus can only have slopes $1/k$, $k\in\Z$ \cite[Corollary~2.5]{os}.
\end{coro}

We also note
\begin{coro}
Up to equivariant isotopy through real tight structures relative to the boundary 
there is a unique $c_i$-real tight contact structure, $i=1,2,3,4$,
on $S^1\times D^2$ with  convex boundary with $s=-1$ and $\#\Gamma=2$.
\end{coro}

\begin{proof}
The case of $i=3$ is shown in \cite{os}.  Indeed similar idea applies to the other cases.
\end{proof}

\begin{remark}
\label{artieksi}
There is a subtility in this uniqueness to be emphasized. For $j=2,3$ or 4, let $T$ be a $c_j$-invariant convex torus in standard form in a contact 3-manifold.
There are exactly two possible sign configurations on the regions of $T$ determined by the linear dividing set on $T$.  
A priori these need not be $c_j$-equivariantly contact isotopic, although they are in non-real setting.
However one can explicitly see that there is indeed  a $c_j$-equivariant contact isotopy between the two sign configurations. In this way the hypothesis 
of \cite[Corollary~3.5]{os} is satisfied so that we may employ Giroux's Flexibility in the real setting and get  the freedom to choose signs 
in one way or the other. Compare  with Remark~\ref{artieksic1}.
\end{remark}

\subsection{$c_1$-real tight solid tori}
\label{c1soy}

First we refine  \cite[Theorem~1.2]{oz}. Let $\gamma$ be a Legendrian arc in a tight 3-ball. We denote the {\em number of half twists along} $\gamma$ by $t(\gamma)=\lfloor 2tw(\gamma)\rfloor / 2 \in\Z+\frac{1}{2}$. Here $tw(\gamma)$ is with respect to the standard framing and relative to the end points.

\begin{lemma}
Let $B$ be a $c_{st}$-real tight 3-ball with real points $r_1,r_2$ on $\partial B$ and the real arc  $r_1 r_2$. 
Then up to equivariant contact isotopy {\em fixing $\xi$ at  $r_1$ and $r_2$}, there is a unique such $c_{st}$-real tight 3-ball. 
Each component of the dividing curve on the oriented annulus $\partial B - (\nu(r_1)\cup\nu(r_2))$ has $t(r_1 r_2)+1$ half twists.
\label{balltw}
\end{lemma}
The proof is an appropriate modification of that of \cite[Theorem~1.2]{oz}. It follows simply from the existence of a standard equivariant contact neighborhood $N$ of the real arc $r_1 r_2$ with a fixed twisting  (\cite[Corollary~2.5]{oz}) (with the extra condition that the contact structure is fixed in a neighborhood of $r_1$ and $r_2$) and the nonexistence of a $c_3$-real tight solid torus $B-N$ with boundary slope $l\in\Z$, except $l=\pm 1$ in which case there is a unique such solid torus.
(\cite[Theorem~4.1]{oz}).

These said, we label the real points  on the fixed boundary of a $c_1$-solid torus as $r_1,\ldots,r_4$, with $r_1 r_2$ and $r_3 r_4$ being the real arcs.

\begin{prop}
There are exactly two $c_1$-real tight solid tori with slope $1/k$, $k\in \Z$, up to equivariant isotopy. These are distinguished by the sign decoration on the boundary or equally on a horizontal disk. For both tori, the real arcs $r_1r_2$ and $r_3r_4$ have $t=-1$. 
\label{eksikey}
\end{prop}
\noindent {\it Proof:}
Since $\tau_{\mu}$ can be made $c_1$-equivariant, it suffices to prove the claim for slope $-1$. 
Let $S$ be such a solid torus with linear, equivariant dividing set $\Gamma$ of  two connected components  and with the slope of the Legendrian ruling made 0. Note that 
$\Gamma$ contains the real points and recall that there are two possible sign decorations on the regions $\partial S - \Gamma$.
We take an equivariant pair of disjoint, meridional disks $D$ and $D'$ in $S$ with Legendrian boundary and perturb them convex equivariantly relative to the boundary so that they reside away from the real points. The sign decoration on  $D$ and $D'$ (once oriented) are determined by that on $\partial S$ (and vice versa). A sufficiently small open neighborhood $\nu(D)$ of $D$ is isotopic to the standard tight  3-ball.
Excising $\nu(D)$ and its disjoint copy $\nu(D')$ from $S$, we get two
$c_{st}$-equivariant  3-balls $B_1$ and $B_2$ with connected dividing sets and a prescribed sign decoration on their boundary.
  
Up to equivariant contact isotopy fixing the contact structure on $\partial B_1$
there is a unique  $c_{st}$-real tight contact structure on $B_1$. It follows from Lemma~\ref{balltw} that  $t(r_1r_2)+1=0$ so that  $t(r_1r_2)=-1$ (with respect to the horizontal framing).   The same discussion applies for $B_2$.  Hence there are at most two $c_1$-real tight structures,  distinguished by the sign decoration on $D$.
These are in fact $\pm\eta_{-1}$ that appear in Theorem~\ref{eqLeg}.(1). 
Note that $c_2^*\eta_{-1}= -\eta_{-1}$.
\hfill $\Box$

\begin{prop}
Let $S$ be a $c_1$-real tight solid torus with real arcs $r_1r_2$ and $r_3r_4$ and with convex boundary on which the equivariant dividing set has exactly two components and slope $-m\in\Z^-$. The number of such $c_1$-real tight solid tori up  to equivariant isotopy relative to the fixed boundary is at most the $(m-1)^{st}$ Catalan number $C_{m-1}$.
These are distinguished by the configuration of the dividing set on a horizontal disk. For all, the real arcs $r_1r_2$ and $r_3r_4$ have $t=-1$
\label{eksiem}
\end{prop}
\noindent {\it Proof:}
We follow the notation in the previous paragraphs. The first paragraph in the previous proof applies verbatim here. Having said that, we investigate the possible dividing set configurations on $D$, which consists of properly embedded $m$ arcs, connecting $2m$ fixed points on $\partial D$. The dividing set configuration must be counted up to rotation. So one boundary parallel arc connecting neighboring points  can be fixed. Then the count is 
the number of expressions containing $m$ pairs of parentheses which are correctly matched; this number is nothing but $C_{m-1}$.
A sample is as in Figure~\ref{catalan}. 
Again up to equivariant contact isotopy fixing that boundary there is a unique  $c_{st}$-real tight contact structure on $B_1$. Whatever the dividing set configuration on $D$ (and its copy on $D'$) is, the contribution to the longitudinal component of the resulting dividing set on the $c_3$-real tight solid torus $B_1 - N$ is zero. Then Lemma~\ref{balltw} assures that $t(r_1r_2)=-1$. Similar conclusion applies for $B_2$.
\hfill $\Box$
\begin{figure}[h]
\begin{center}
\resizebox{7cm}{!}
{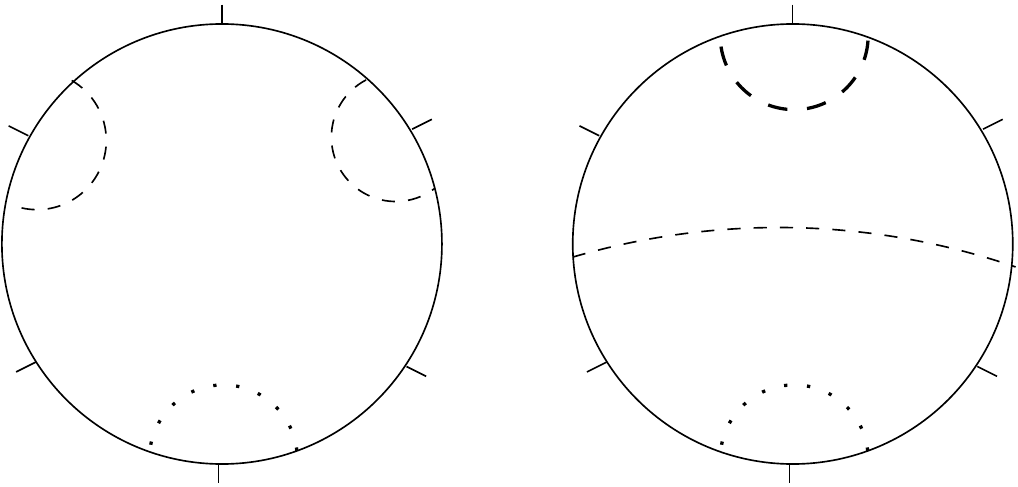}
\caption{Possible dividing set configurations on a horizontal disk for $m=3$; here $C_2=2$.}
\label{catalan}
\end{center}
\end{figure}

One would attempt to get an alternative upper bound for the count in Proposition~\ref{eksiem} as follows. Instead of taking a convex $D$ and its equivariant  copy $D'$ one might have considered a pair of horizontal $c_1$-invariant disks $D_{12},D_{34}$ containing the real arcs $r_1r_2$ and $r_3r_4$ respectively. The problem here is that it is not apriori true that $D_{12}$ or $D_{34}$ can be perturbed convex equivariantly. This would be possible if and only if the Legendrian boundary of each half disk (for example one of the connected components of $D_{12} - r_1r_2$) has negative twisting with respect to the half disk, which is not for granted. Nevertheless if one assumes convexity, it can be observed that the count is still large and is related to the meandric numbers.

Another method to find an upper bound for the count could be the slicing operation in \cite{ho}. However the intriguing phenomenon  here
is that  slicing does not work $c_1$-equivariantly, simply because of the nonexistence of $c_1$-real tight basic slices.

\begin{prop}
There is no $c_1$-real tight basic slice.
\label{nobasic}
\end{prop}
\noindent {\it Proof:}
Consider a basic slice $R$ with a $c_1$-real contact structure and with invariant convex inner and outer boundary tori  $T_1$ and  $T_2$ respectively.
Without loss of generality, we may suppose that $T_1$ and $T_2$ are in  $c_1$-equivariant  standard form with boundary slopes $-1$ and $-2$ respectively and slopes of  Legendrian ruling 0 on both.

Take an equivariant pair of disjoint meridional annuli with Legendrian boundary.
Perturb the pair equivariantly relative to their boundary to obtain two convex annuli $A$, $A'$. The dividing set $\Gamma_A$ on $A$ has one boundary parallel arc with end points on $T_2$ and two arcs from $T_1$ to $T_2$. The dividing set $\Gamma_{A'}$  is a symmetric copy of $\Gamma_A$.

Keeping track of the dividing set on the boundary of any of the $c_1$-real solid tori bounded by $A$, $A'$, $T_0$ and $T_1$ (after edge rounding), the  possible configurations for the dividing sets all result with an overtwisted solid torus 
(See Figure~\ref{yok12}). \hfill $\Box$

\begin{figure}[t]
\begin{center}
\resizebox{10cm}{!}
{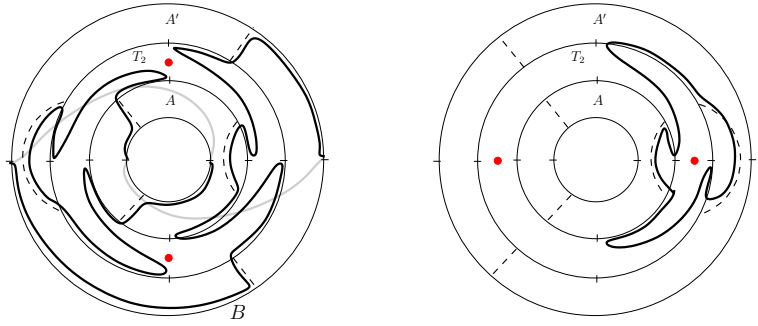}
\caption{Both dividing set configurations on $A_0$ and $A'_0$ result with OT solid tori. The innermost and outermost annuli are $A$ and $A'$. The annulus in between is a portion of the outer boundary $T_2$ where the dividing curves rotate $-1/4$ turn. The annulus that is a part of $T_1$ is invisible in the figure; on it the dividing curves rotate $-1/2$ turn.}
\label{yok12}
\end{center}
\end{figure}

Now, we prove the following proposition for a  genuine equivariant  double slice, that is, 
one which cannot be peeled-off into two equivariant basic slices with convex boundary.

\begin{prop}
There are exactly two  genuine $c_1$-real tight double slices up to equivariant isotopy relative to the equivariant  convex boundary.  They are distinguished by the sign decoration on the boundary.
\label{vardouble}
\end{prop}
\noindent {\it Proof:}
We  assume without loss of generality  that the inner and outer boundary slopes of the double slice $M$
are $-1$ and $-3$   respectively and that the 
Legendrian ruling has slope $0$ on both. Such a $c_1$-real tight genuine double slice indeed exists. It is nothing but a union of two usual basic slices of Honda, the by-pass disks of which are placed symmetrically. 

To count them, we proceed as in the previous proof. Keeping track of the dividing set on the boundary of any of the $c_1$-real solid tori bounded by $A$, $A'$, $T_0$ and $T_1$ (after edge rounding), the  possible configurations for the dividing sets all result with an overtwisted solid torus except the single one up to rotation (See Figure~\ref{var13}). There we get a $c_1$-tight 
solid torus with boundary slope $-1$.  There are  two such tori up to equivariantly contact isotopy 
relative to the boundary. These structures can be distinguished by the sign decoration on the boundary (or equivalently on, say, $A$).

\begin{figure}[t]
\begin{center}
\resizebox{5cm}{!}
{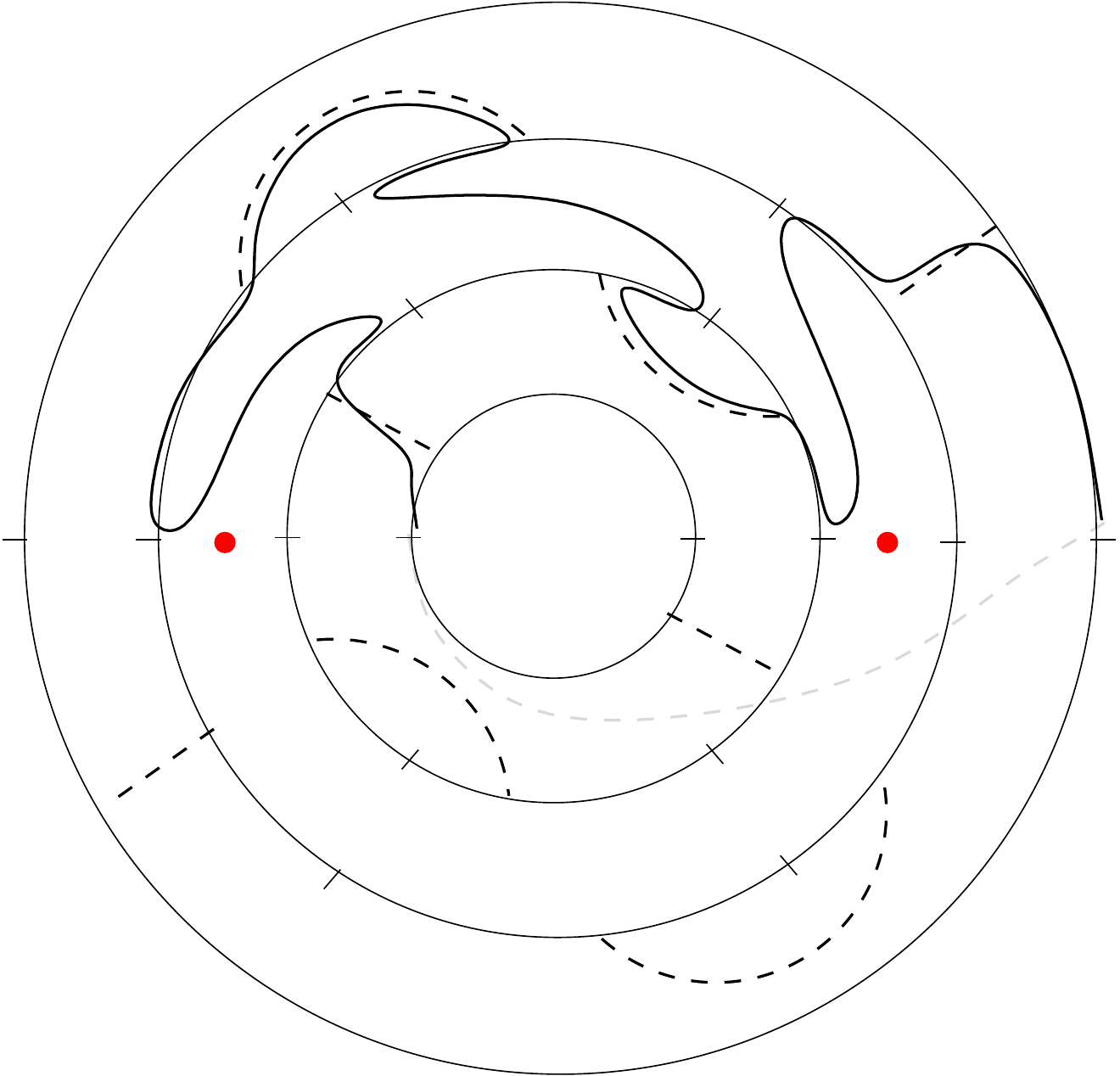}
\caption{The only configuration on $A$ and $A'$ that does not lead an OT disk. It gives a $c_1$-real $(-1)$-solid torus.}
\label{var13}
\end{center}
\end{figure}

Having said these, fixing any of the two sign decorations on the boundary of the double slice, there is at most one $c_1$-real tight structure on the
double slice, obtained by gluing the  two $c_1$-real tight solid tori of slope $-1$.
 \hfill $\Box$

The discrepancy between the contact topological and the  $c_1$-equivariant cases is that in the former  inside a double thick torus there are convex tori of every intermediate slope \cite[Proposition~4.16]{ho}, which is not necessarily true in the latter. 
Let us note explicitly the following corollary of the two preceeding propositions.

\begin{coro}
$c_1$-equivariant bypass attachments come in $c_1$-equivariant pairs. There is no $c_1$-equivariant bypass attachment in a $c_1$-real tight genuine double slice. 
\end{coro}

\section{Classification on $S^3$ and $\R P^3$}
\label{s3rp3}
\subsection{Real tight $S^3$}

\begin{theorem}
Up to isotopy through real tight structures there is a unique real tight $S^3$ with nonempty real part.
\label{S3}
\end{theorem}

\noindent {\it Proof:}
Recall that  up to isotopy through real structures there is a unique real structure $c_{S^3}$ on $S^3$ 
where the fixed point set is  an unknot $L$.  Let $H$ be an equivariant contact neighborhood of $L$. By Theorem~\ref{eqLeg}(2),
$H$ is isotopic to the real tight solid torus $\eta_{\,tb(L)}$
with $tb(L)\in \Z^-$.   
The complement of $H$ in $S^3$ is  a $c_{3}$-real tight solid torus with convex boundary of slope $tb(L)$.  
Thus by Theorem~\ref{anadolu} and Theorem~\ref{eqLeg}(3),  $tb(L)$ must be $-1$ and the proof follows.
 \hfill $\Box$

In the above proof  we use a genus-1 Heegaard splitting of $S^3$.  A similar proof can be 
produced using a genus-0 Heegaard splitting, in which case we apply \cite[Theorem~2.1]{os} twice.

\subsection{Real tight $\R P^3$} 
Let us consider $\R P^3$ as the link manifold $X_1^+$ of the $A_1$ singularity $\{x^2 - y^2 + z^2=0\}\subset \C^3$,
i.e. 
$$
X_1^{+}=S^5_{\varepsilon} \cap \{x^{2} - y^2 + z^2=0\}
$$
where $S^5_{\varepsilon}$ is a small 5-sphere  in $\C^3$ centered at 0 with radius 
$\varepsilon>0$. $X_1^+$ is naturally a real tight 3-manifold with the real and contact  structures being the standard ones induced respectively by the complex conjugation  and the standard symplectic structure on $\C^3$. 
We recall that  up to contact isotopy there is a unique tight structure on $\R P^3$ \cite{ho} and up to isotopy through real structures there is a unique real structure on $\R P^3$ with nonempty real part \cite{hr}. 

\begin{theorem}
\label{rp3}
Up to isotopy through real tight structures there is a unique real tight $\R P^3$ with nonempty real part. Any such is equivariantly contact isotopic to  $X_1^+$.
\end{theorem}
\noindent {\it Proof:}
The real structure $c$ may be described as follows  \cite{hr}: in the genus-1 Heegaard splitting of $\R P^3$ let $c$ act on the solid tori $H_1$ and  $H_2$ as $c_2$ and
$c_4$ respectively. (This is a type-$B$ real structure. See the next section for the definition.)
The fixed point set of $c$ is the union of the  core circles $L_i\subset H_i$
and the gluing map is $\begin{pmatrix} -1 & 0 \\ 2 & 1 \end{pmatrix}$. 
We may suppose that $H_1$ is the standard $c_{2}$-real contact neighborhood of $L_1$. 
Thus the dividing curve on $\partial H_1$ has slope $s_1=1/m$, $m\in\Z$, by Theorem~\ref{anadolu}, so that the slope on $\partial H_2$  is  $s_2=-(2m+1)/m$ too.
By Corollary~\ref{egimler} the only value for such a slope on the $c_4$-real $H_2$
is $\infty$ and $m=0$. There is a unique such $H_2$ which is $c_{4}$-real tight; that is the standard  $c_{4}$-real contact neighborhood of $L_2$. Hence there is at most 1 real tight $\R P^3$.  Of course $X_1^+$ is such a manifold.  
\hfill $\Box$

Let us observe that in this unique real tight $\R P^3$, the rational Thurston-Bennequin numbers $tb_{\Q}(L_1)$ and $tb_{\Q}(L_2)$ are both  $-1/2$; to see this it is enough to note that  the rational Seifert framings along $L_1$ and $L_2$ are given by $(-1,2)$ on both $\partial H_1$ and 
$\partial H_2$. This observation is in accordance with \cite{oz} and \cite{be}.

Recall that for $L(2,1)$, a type-$B$ real structure is equivariantly isotopic to a type-$A$. (See the next section.) Thus one could attempt to produce an alternative proof using a  genus-1 Heegaard splitting using $c_1$-equivariant solid tori. Again $s_1$ and $s_2$ would both be equal to $\infty$ or $-1$. Since both cases may occur in principle, we would conclude that
there are at most two real tight structures on $\R P^3$, which is a weaker result than the above.

\section{Real tight lens spaces of types $A$ and $B$} 
\label{AveB}

Here we start to classify the real tight structures on lens spaces $L(p,1)$ and $L(p,p-1)$  with $p>2$ and with nonempty real part.

Let $q=1,p-1$ and $M=(L(p,q),c)$ be  a real lens space with nonempty real part. We 
consider a genus-1 Heegaard decomposition of $M$ with handlebodies $H_1,H_2$.
By definition this decomposition is a $c$-real Heegaard decomposition  if and only if $c$ exchanges the handlebodies, i.e. $c$ is isotopic to a real structure of type $C$ or $C'$ in the notation of \cite{hr}. (See \cite{os} for an introduction of  real Heegaard decompositions.) In that case the gluing map is an orientation reversing involution of $\partial H_1$ hence is strongly equivalent to a linear involution 
$\phi=\pm\Phi$ where $\Phi=\begin{pmatrix} -q & q' \\ p & q\end{pmatrix}$; the sign $+$ for type C  and $-$ for type C$'$ \cite[Lemma~4.6]{hr}.

If $c$ maps each handlebody to itself then the pair $(c|_{H_1},c|_{H_2})$ is  isotopic  through real structures to $(c_1,c_1)$, $(c_2,c_3)$ or $(c_2,c_4)$  which correspond to the real structures of type $A$, $B$ and $B'$ respectively. For $L(2,1)$, all of these 5 types of real structures are isotopic through real structures. For $L(p,1)$, $p>2$, types $A$ and $C$ are equivalent and $B,B',C'$ are equivalent. For $L(p,p-1)$, $p>2$, types $A,C'$ are equivalent and $B,B',C$ are equivalent. 

Letting $X$ stand for one of the types $A$, $B$, $B'$, $C$ or $C'$, we denote by $l_X(p,q)$ the number of real tight structures on the real lens space $L(p,q)$ of type $X$ up to isotopy through real tight structures. 

Our main classification results for real tight lens spaces are

\begin{theorem}
\label{hudutB}
(a) $l_{B}(p,q) > 0$ if and only if $q=1$ or $q=p-1$.  \\
(b) $l_{B}(p,p-1) = 1$. This unique real tight structure is given by the isolated surface singularity $A^+_{p-1}$. \\
(c) $l_{B}(p,1) = 1$ for $p$ even; $l_{B}(p,1) \leq 1$ for $p$ odd.
\end{theorem}

\begin{theorem}
\label{hudutA}
(a) $l_{A}(p,p-1)\geq 1$. The existent real tight structure is given by the isolated surface singularity $A^-_{p-1}$. \\
(b) $p-1\leq l_{A}(p,1)\leq C_{p-2}+1$.
\end{theorem}

Here $A^{\pm}_{p-1}=\{\pm x^{p} - y^2 +  z^2=0\}\subset \C^3$.

We prove these theorems by first justifying in Section~\ref{ust} the upper bounds involved in the given (in)equalities and then in Section~\ref{cerrah} we explicitly construct distinct real tight lens spaces via contact equivariant surgery to prove the validity of lower bounds. See Figure~\ref{tumtablo} for the upper and lower bounds and the claimed equivariant surgery diagrams for each case. 

In Section~\ref{realob} we construct real open books for the case $q=1$. That construction displays in particular real contact Heegaard decompositions of higher genera.

Following that observation we ask in Section~\ref{CandC'} if there is any real contact Heegaard decomposition of genus~1 at all. We answer that question negatively for some cases and for the remaining cases we give upper bounds for their count.

In Section~\ref{q=p-1} we show that all  the examples for the case $q=p-1$ in Figure~\ref{tumtablo}  are real Milnor fillable.

\begin{figure}[t]
\begin{center}
 \includegraphics[width=12.5cm]{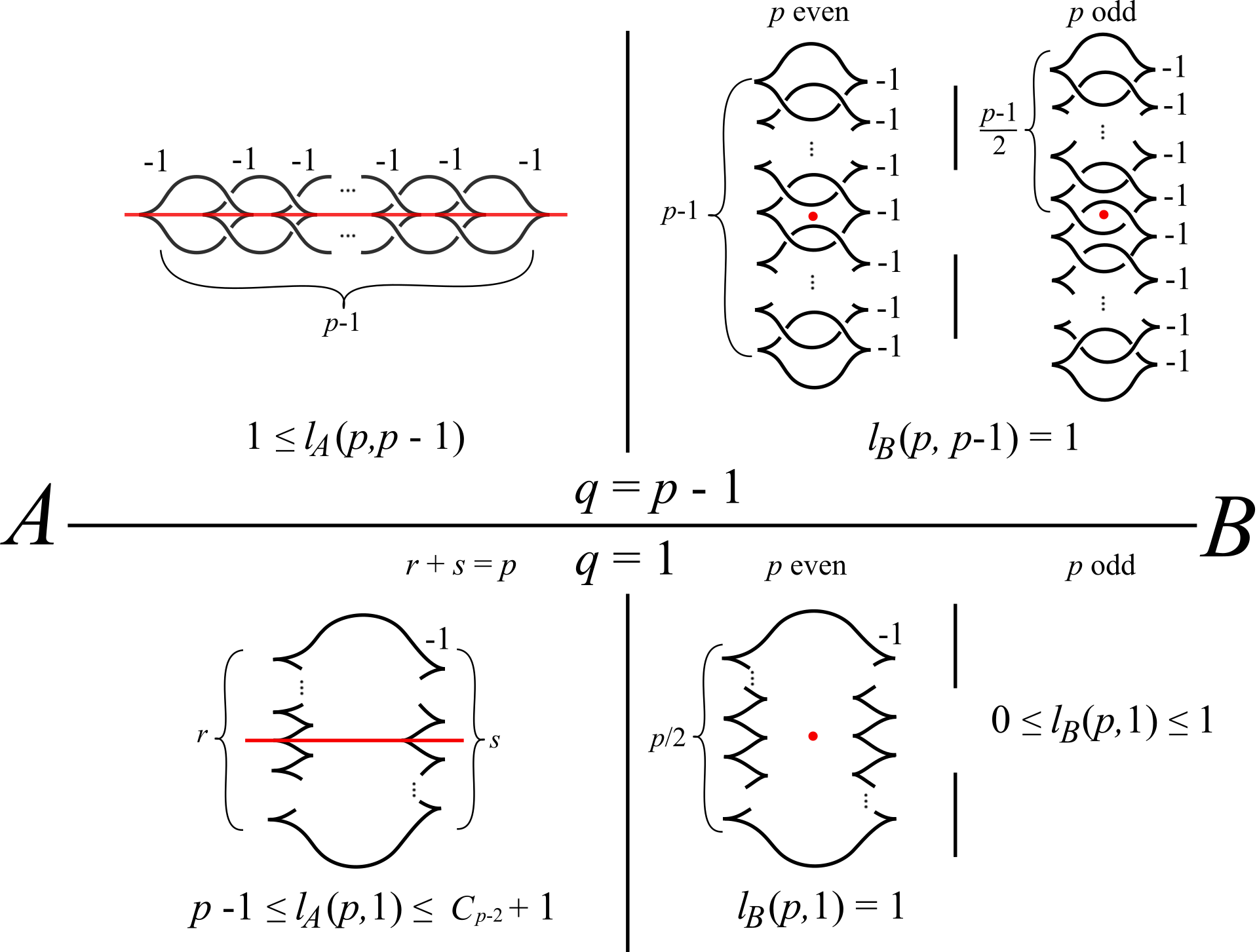}
 \caption{}
\label{tumtablo}
\end{center}
\end{figure}

\subsection{Upper bounds.}
\label{ust}

Here we prove that the upper bounds stated in Theorems~\ref{hudutB}-\ref{hudutA} are valid. 
First we do that for Theorem~\ref{hudutB}.

\begin{lemma}
\label{typeB}
For $q=1,p-1$, $l_B(p,q)\leq 1$. Otherwise $l_B(p,q)=0$.
\end{lemma}
\noindent {\it Proof:}
Assume $p>2$.
Similar to the proof of Theorem~\ref{rp3} we consider the real decomposition $(H,c_2; H',c_j)$ for the lens space $L(p,1)$.  If $p$ is even then $j=2$, otherwise $j=3$ or $4$.
By  Theorem~\ref{anadolu} we have $s=1/k$, $k\in\Z$ for the slope of $H$ so that $s'=-(kp+1)/k$. Again by Theorem~\ref{anadolu} this latter slope is possible only if 
$kp+1=\pm 1$ or $k=0$. Since $p\neq 2$, the only possible solution is when $k=0$ and $s=s'=\infty$. 
For $p$ even such $H'$ is unique by Theorem~\ref{anadolu} and Remark~\ref{artieksi}.
For $p$ odd, since there is no $c_3$-real tight solid torus with slope $\infty$, $j$ must be 4 and such $H'$ is unique.

As for the lens space $L(p,p-1)$, the same argument gives us the slope $s_2=(kp+p-1)/(kp+p-k-2)$ and the single possible situation, occurring for $kp+p-1=\pm 1$.
This has a unique solution $k=-1$ and that holds for any $p$. Thus we obtain $s=s'=-1$.  
If $p$ is even, then $j=2$ and there is a unique such $H'$ with slope $-1$.
If $p$ is  odd,  since there is no $c_4$-real tight solid torus with slope $-1$,  $j$ must be 3 and there is a unique such $H'$.

With similar considerations one can easily see that in order to get $s_1,s_2\in\{0,1/k:k\in\Z \}$ $q$ must be either $1$ or $p-1$. 
\hfill $\Box$ \\

The following lemma justifies the upper bound in Theorem~\ref{hudutA}(b).

\begin{lemma}
\label{typeA}
$l_{A}(p,1)\leq C_{p-2}+1$.
\end{lemma}
\noindent {\it Proof:}
We consider the decomposition $(H,c_1; H',c_1)$ for the lens space $L(p,1)$. We assume $s=1/k$, $k\in\Z, k\leq 0$.  
Recall that $H$ can be thickened  to a solid torus of slope $\infty$ by simply peeling off basic slices from $H'$ (see the proof of \cite[Proposition~4.17]{ho}). By Proposition~\ref{vardouble} and its corollary, two consecutive basic slices can be arranged to form a $c_1$-equivariant geniune double slice. Thus in the $c_1$-equivariant case,  $H$ can be thickened  to a $c_1$-equivariant solid torus of slope either $\infty$ or $-1$ depending on 
$k$ being even or odd respectively, leading $s'=\infty$ or $p-1$ respectively. In the former case, there is a unique $c_1$-real tight $H'$ relative fixed boundary (Proposition~\ref{eksikey})  while in the latter there are $C_{p-2}$ (Proposition~\ref{eksiem}).
Hence $C_{p-2}+1$ is an upper bound.
\hfill $\Box$ \\

\subsection{Lower bounds: construction via equivariant contact surgery.}
\label{cerrah}

Here we show that the lower bounds in Theorems~\ref{hudutB}-\ref{hudutA} are valid. We do that by constructing explicitly distinct examples via equivariant contact surgery.

Let $K$ be a $c_1$-invariant Legendrian knot in the standard real tight $S^3$. On the boundary of a standard equivariant contact neighborhood $U$ of $K$, let us denote the longitude determined by the contact framing by $\lambda$ and a meridian by $\mu$, as usual. With respect to $\lambda$ the slope on $\partial U$ is $\infty$. For an equivariant $+1$- and respectively $-1$-surgery along $K$, the $c_1$-tight solid torus to be glued back must have the boundary slope $+1$ and $\infty$ respectively, both of which exist (Proposition~\ref{eksikey}).
This discussion proves
that $c_1$-real contact $(\pm 1)$-surgery exists and is well-defined. In the notation of  \cite{ose} this is an equivariant 
surgery of type $1_1$, gluing back a $c_1$-solid torus. (For a more general discussion see \cite{ose}.)

It follows that the $c_1$-equivariant contact surgeries  in Figure~\ref{tumtablo} (and in Figures~\ref{tabloremark},~\ref{L511}-\ref{L513}) 
are well-defined. We would like to emphasize here that in each of those diagrams, since the real involution exchanges the signs on the boundary of a standard
neighborhood of the unknots and since the involution is equivariantly isotopic to the identity, the sign choice does not alter the final manifold up to $c_1$-equivariant contact  isotopy. 

Similarly let $K$ be a $c_4$-invariant Legendrian knot in the standard real tight $S^3$. A contact $\pm 1$ equivariant surgery along $K$, exists and is unique.  This is an equivariant  surgery of type $4_2$, gluing back a $c_2$-solid torus \cite{ose}.

If $K$ and $K'$ constitute an equivariant pair of disjoint Legendrian knots, an equivariant pair of contact surgeries
performed along $K$ and $K'$ is always possible. The upper right diagrams in Figure~\ref{tumtablo} contains such. \\

\noindent {\em Proof of lower bounds in Theorems~\ref{hudutB}-\ref{hudutA}.} 
Recall that the only real structure on a lens space having quotient $S^3$ is of type~$A$ \cite{hr}.
Hence 
all the diagrams in Figures~\ref{L511}-\ref{L513} determine real tight structures of type~$A$, justifying the bottom left part of Figure~\ref{tumtablo}. Their count proves the lower bound in Theorem~\ref{hudutA}(b).

Similarly,  the types of real structures obtained via the remaining equivariant contact surgery diagrams is determined as seen in Figure~\ref{tumtablo}.
Their existence proves all the remaining lower bounds. 
\hfill $\Box$

\begin{remark} Note that the contact surgery diagram in Figure~\ref{tabloremark} determines a real tight structure of type $A$ on $L(p, p-1)$ too when $p$ is even. We have not any other tool yet to distinguish this real tight structure from that given in top left diagram in Figure~\ref{tumtablo}. 
\end{remark}

\begin{figure}[h!]
\begin{center}
\includegraphics[width=3cm]{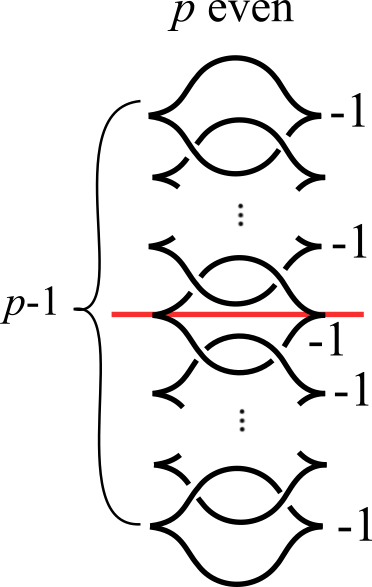}
\caption{ Type $A$ real tight structure on $L(p, p-1)$, $p$ even.}
\label{tabloremark}
\end{center}
\end{figure}

\section{Construction via real open books for $q=1$} 
\label{realob} 
In this section we first recall the definition of a real open book decomposition \cite{os2}. Then we explicitly construct real open books for the real tight contact structures  on lens spaces $L(p,1)$ with $p>2$ and with nonempty real part. 

An abstract real open book decomposition is a triple $(S, f, c)$ where $S$ is a compact surface with boundary, $c$ a real structure on $S$ and $(S,f)$ is an abstact open book with monodromy $f$ satisfying $f \circ c = c\circ f^{-1}$. A positive real stabilization of the abstract real open book $(S,f,c)$ is the abstract real open book $(S', \tilde{f} \circ \sigma, \tilde{c} \circ \sigma)$ where $S' = S \cup H$ and $H$ is either a 1-handle with $c$-invariant attaching region or a union of two 1-handles with their attaching regions interchanged by $c$. The new monodromy $\tilde{f}$ is the natural extension of $f$ over $S'$. Positive real stabilizations come in types I-IX (see \cite{os2}).

Any of the $p-1$ tight contact structures on $L(p,1)$ can be obtained via a Legendrian surgery, i.e. a contact $(-1)$-surgery along a Legendrian unknot $L$ with  ${tb}(L) = p-1$. The diagrams in Figures~\ref{L511}-\ref{L513} reveal that those surgery operations can be made equivariant. As noted at the end of Section~\ref{cerrah} they all give type-$A$  real structures.
Here we construct real open book decompositions $(S, f_i, A)$ 
(with abuse for $A$ in the notation) for 
each of those real tight lens spaces $(L(p,1), \xi_i, A)$, $i=1, \ldots, p-1$. The real 
pages are to be disks with $p-1$ holes. For this we start with the real open book  
$(D, {\tt id}, c)$ of the real tight $(S^3, \xi_{st},c_{st})$ where the real page $D$ is a disk (Figure~\ref{S3ob}). We then perform positive real stabilizations (when necessary) in order to obtain the page on which the required Legendrian knot would reside. 
\begin{figure}[h!]
\begin{center}
\includegraphics[width=2cm]{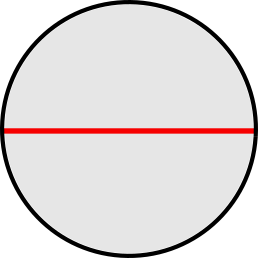}
\caption{Real open book $(D, {\tt id}, c)$ of the real tight  $(S^3, \xi_{st},c_{st})$.}
\label{S3ob}
\end{center}
\end{figure}

\subsection{Real open books for $L(5,1)$.} \label{realob51} As an illustration of the general case  we construct real open book decompositions $(S, f_i, A)$ for each real tight $(L(5,1), \xi_i, A)$, $i=1, \ldots, 4$, where the page $S$ is a disk with 4 holes. 

$L(5,1)$ has two universally tight contact structures. Surgery diagram for one of the universally tight contact structures, say $\xi_1$, and the real open book supporting the real tight $(L(5,1), \xi_1, A)$ are given in Figure~\ref{L511}. The monodromy $f_1$ of the real open book is $f_1= t_{a_1}t_{a_2}t_{a_3}t_{a_4}t_L$ where each $t_{a_i}$ is a Dehn twist along the belt circle $a_i$. Note that in this case we perform two sets of positive real stabilizations of type~III (i.e. two pairs of symmetric stabilizations; see \cite{os2}) on the real open book $(D, {\tt id}, A)$ of the real tight $S^3$. 
\begin{figure}[h!]
\begin{center}
\includegraphics[width=5cm]{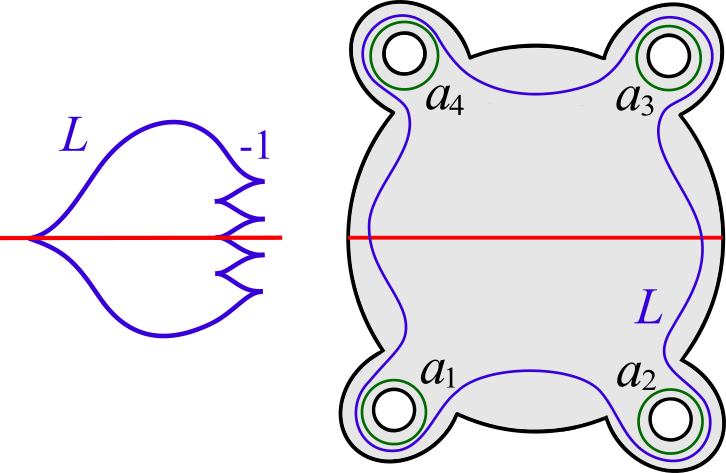}
\caption{Real open book  $(S, f_1, A)$ of real tight $(L(5,1), \xi_1, A)$.}
\label{L511}
\end{center}
\end{figure}

Surgery diagram for the other universally tight contact structure, say $\xi_4$, on the lens space $L(5,1)$ and the real open book  supporting the real tight $(L(5,1), \xi_4, A)$ are given in Figure~\ref{L514}.  Note that in this case  we perform first two positive real stabilizations of type~II on the real open book $(D, {\tt id}, c)$ of the real tight 
$S^3$. Next, we add a type-III positive real stabilization to this open book. The monodromy $f_1$ of the real open book is $f_4= t_{a_1}t_{a_2}t_{a_3}t_{a_4}t_L$ where the curves $a_i$ and the Legendrian knot $L$ are as in Figure~\ref{L514}.

\begin{figure}[h!]
\begin{center}
\includegraphics[width=5cm]{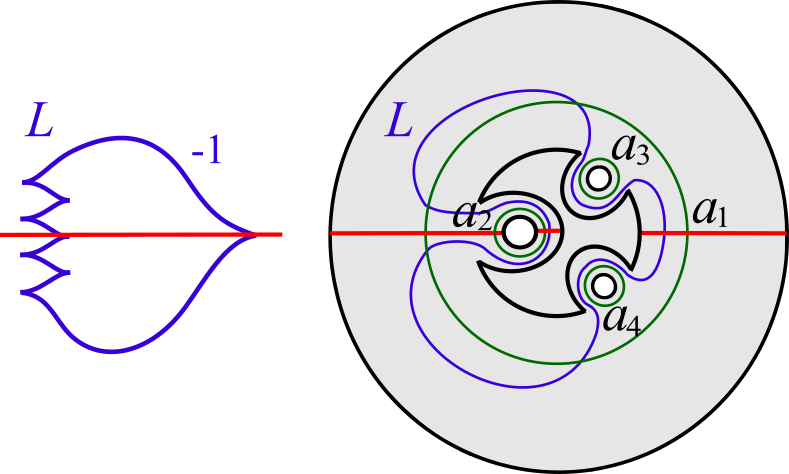}
\caption{Real open book $(S, f_4, A)$ of real tight $(L(5,1), \xi_4, A)$.}
\label{L514}
\end{center}
\end{figure}

The real open books for the remaining two virtually overtwisted contact structures, say $\xi_2$ and $\xi_3$, are given in Figure~\ref{L513}. Note that in each of these cases the real open book is constructed from the real open book in the intermediate stage of the previous paragraph by adding one positive real stabilization of type~III appropriately.

\begin{figure}[h!]
\begin{center}
\includegraphics[width=10cm]{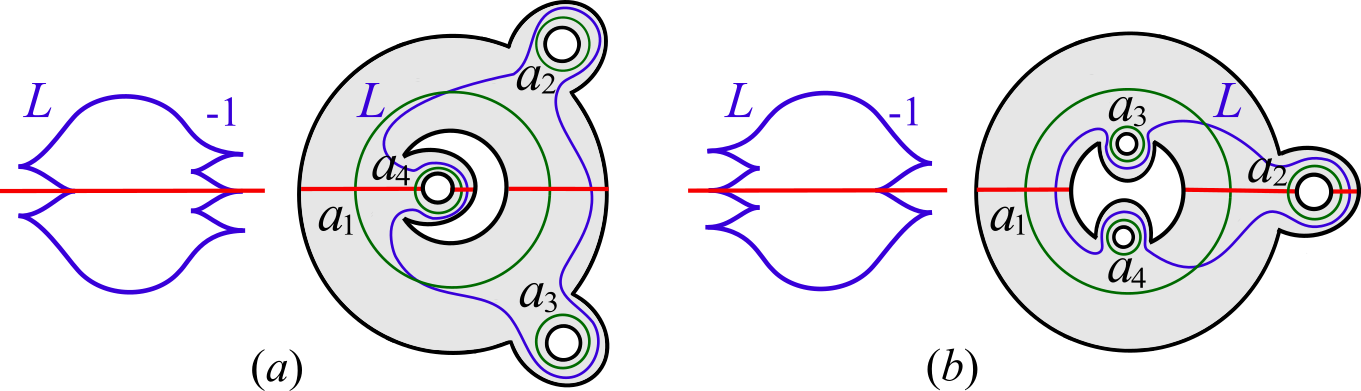}
\caption{Real open book  $(S, f_j, A)$ of real tight $(L(5,1), \xi_j, A)$: $(a) j=2$; $(b) j=3$.}
\label{L513}
\end{center}
\end{figure}

\subsection{Real open books for $L(p,1)$ for $p$ odd.} The construction in the preceding section generalizes in a straightforward manner to the lens spaces $L(p,1)$, $p$ odd. The real open books supporting the two real universally tight structures are given in Figure~\ref{Lp1o}. The real open book on the left (resp. on the right) corresponds to the real universally tight contact structure coming from a Legendrian surgery along a Legendrian knot having all its $p-2$ zigzags on its right (resp. on its left). In the former case all real stabilizations are of type~III while in the latter case two type-II positive real stabilizations are needed too. 

\begin{figure}[h!]
\begin{center}
\includegraphics[width=6.5cm]{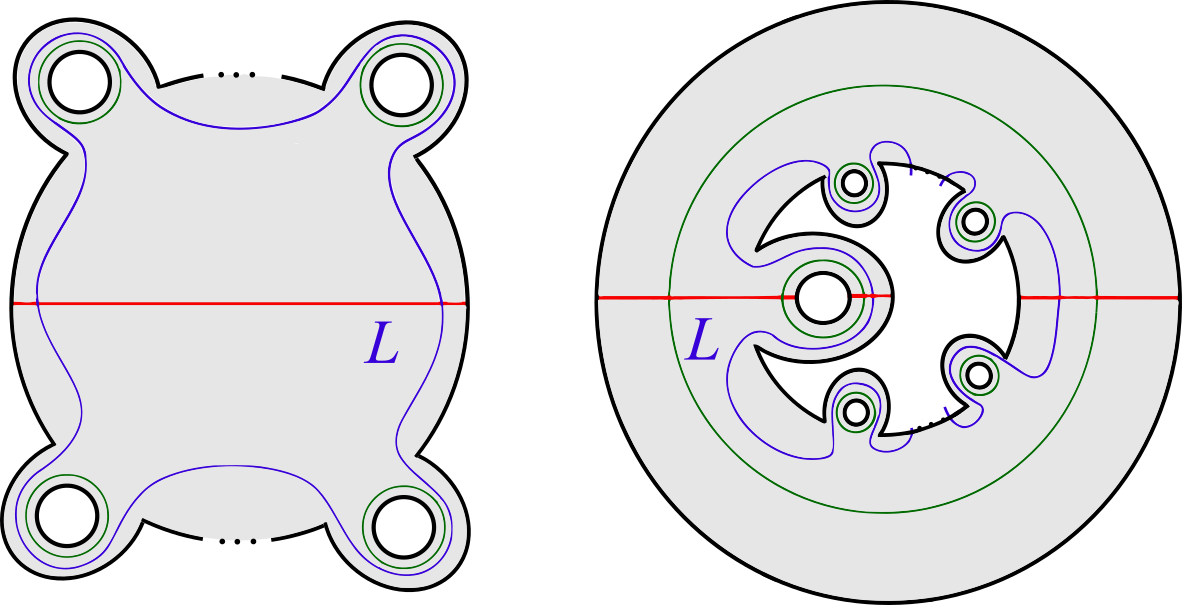}
\caption{Real open book decomposition of the universally tight real $L(p,1)$'s when $p$ is odd. }
\label{Lp1o}
\end{center}
\end{figure}

The real open books of the remaining  virtually overtwisted real tight lens spaces are given in Figure~\ref{Lp1odd}. The real open book on the left of Figure~\ref{Lp1odd} corresponds to a Legendrian knot $L$  where $L$ has odd number of zigzags on its left and even number of zigzags on its right (compare to $(L(5,1), \xi_2, A)$ in Section~\ref{realob51}). The real open book on the right of Figure~\ref{Lp1odd} corresponds to a Legendrian knot $L$ where $L$ has even number of zigzags on its left and odd number of zigzags on its right (compare to $(L(5,1), \xi_3, A)$ in Section~\ref{realob51}).  

\begin{figure}[h!]
\begin{center}
\includegraphics[width=7cm]{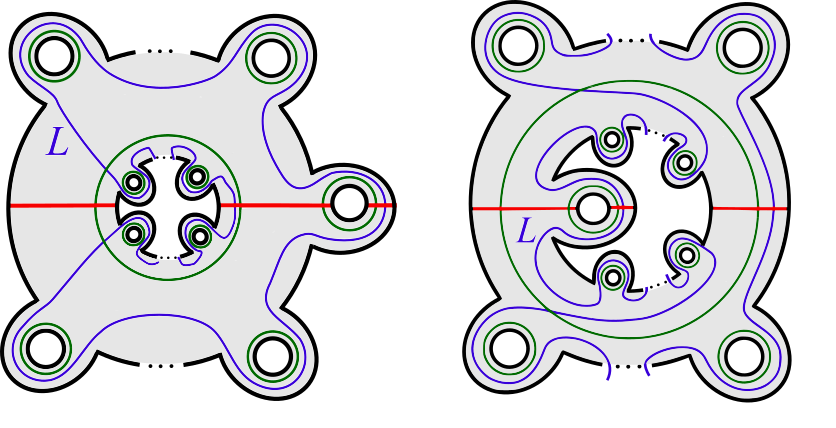}
\caption{Real open book of real virtually overtwisted tight $L(p,1)$, $p$ is odd.  }
\label{Lp1odd}
\end{center}
\end{figure}

\subsection{Real open books for $L(p,1)$ for $p$ even.} The analogue of the construction of real open books for $L(p,1)$ for $p$ odd extends to lens spaces $L(p,1)$ for $p$ even as expected. The real open books  supporting the two real universally tight lens spaces $L(p,1)$ for $p$ even are given in Figure~\ref{Lp1e}. 

\begin{figure}[h!]
\begin{center}
\includegraphics[width=6cm]{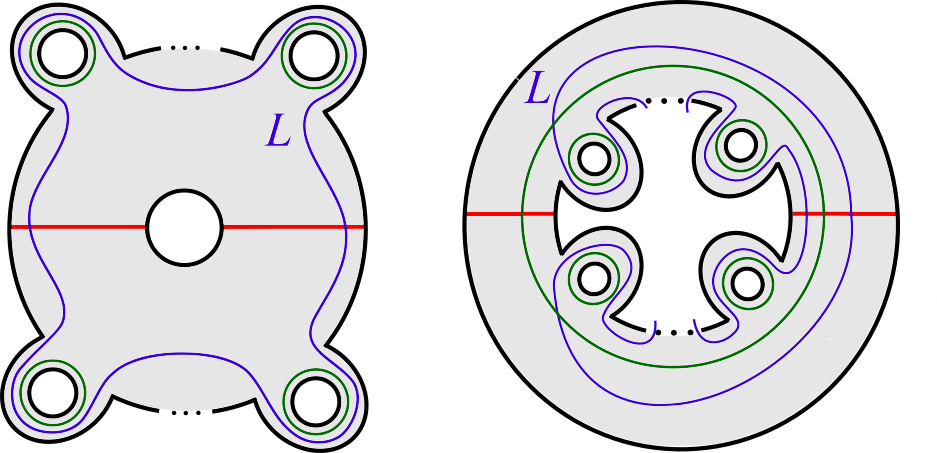}
\caption{Real open book  of the universally tight real  $L(p,1)$'s when $p$ is even.}
\label{Lp1e}
\end{center}
\end{figure}

The real open book  supporting the virtually overtwisted real tight lens spaces $L(p,1)$ for $p$ even is given in Figure~\ref{Lp1ev}. 

\begin{figure}[h!]
\begin{center}
\includegraphics[width=6cm]{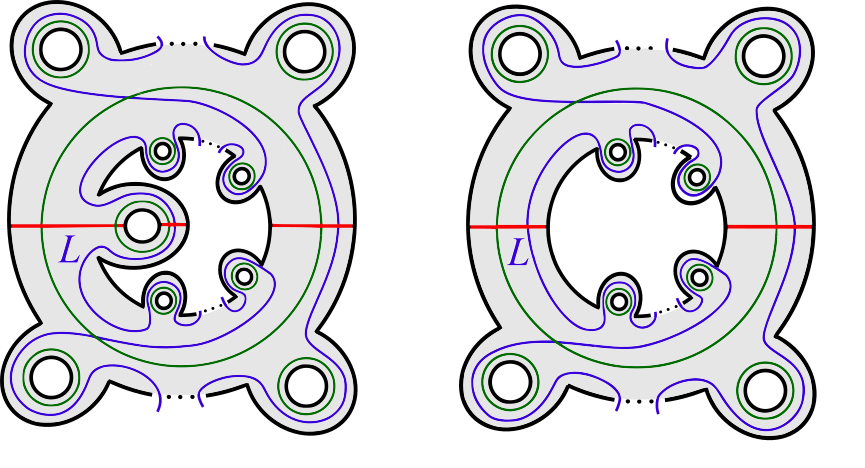}
\caption{Real open book  of real virtually overtwisted tight $L(p,1)$, $p$ is even. The Legendrian knot $L$ on the left/right has odd/even number of zigzags on its left and even/odd number of zigzags on its right.}
\label{Lp1ev}
\end{center}
\end{figure}

\section{Classification for type $C$ and $C'$}
\label{CandC'}
In the previous section, all the real open books that appeared determine real contact Heegaard decompositions of genera greater than 1.
When the lens space is given as a real Heegaard decomposition of genus 1 (i.e. exactly in case of types  $C$ or $C'$),
it might be the case that the Heegaard surface cannot be made equivariantly convex. We denote by  $l^*_X(p,q)$ the number of real tight structures on the real lens space $L(p,q)$ of type $X=C,C'$ up to equivariant contact isotopy with the condition that the lens space assumes a real contact Heegaard decomposition of genus 1.
We ask what $l^*_X(p,q)$ equals for $q=1,p-1$  and if that number is nonzero at all. At the end of this section we come up with a negative result: $l^*_{C'}(p,1)=0$ and $l^*_{C}(p,p-1)=0$ (Theorem~\ref{l*Cp-1}).

We start with determining upper bounds for the remaining cases.

\begin{prop}
\label{typeC}
For $p>2$, \\
(a) $\disp l^*_{C'}(p,p-1)\leq\left\{ \ba{ll}  4, & p \mbox{ even}; \\ 6, &  p \mbox{ odd.} \ea \right. $ \\
(b) $\disp  l^*_C(p,1)\leq\left\{ \ba{ll}  p, & p \mbox{ even}; \\ 2p-2, &  p \mbox{ odd.} \ea \right. $ \\
\end{prop}
\noindent {\it Proof:}
Consider an abstract genus-1 $C$-real Heegaard decomposition $(H,H';c)$ of the lens space $M=L(p,q)$. Set $T=\partial H$.
(Sometimes we view $H$ embedded in $M$.) Since we are interested in computing $l_X^*$, we assume $T$ is convex in $M$.  $T$ may be assumed to be in standard form 
as a convex torus. 
The real part (and the  characteristic foliation) on $T$ and the invariant dividing set $\Gamma$ are  determined as linear sets by the eigenvector of  $\phi$ that corresponds to the eigenvalue $+1$ and $-1$ respectively.
For type $C$ and $q=1$ the $(-1)$-eigenvector  of $\phi$ is $(2,-p)$ and the slope of $T$ is 
$$
s=-p/2=
\left\{ \ba{ll}  \left[-p/2\right], & p \mbox{ even}; \\ 
\left[-(p+1)/2,-2\right], & p \mbox{ odd.} \ea \right.
$$
For type $C'$ and $q=p-1$ the $(-1)$-eigenvector is $(2-p,p)$ and 
$$s=p/(2-p)=\left\{ \ba{ll} -(p/2)/(-1+p/2)=\left[-2,-2,\ldots,-2\right], & p \mbox{ even}; \\ \left[-2,\ldots,-2,-3\right], &  p \mbox{ odd.} \ea \right.$$

The number of connected components of $\Gamma$ is an even positive integer; we denote this number by $\# \Gamma$.
If we show that $\# \Gamma=2$ then the  proof follows. Indeed in that case, for type $C$ and $q=1$ we read off from the slope expressions above that the number of tight contact structures on the solid torus  $H$ up to contact isotopy relative boundary  is $-p/2$ for $p$ even and  $((p+1)/2-1)\cdot 2=p-1$ for $p$ odd \cite{ho}.
For type $C'$ and $q=p-1$, we have  that number   $2$ for $p$ even and $3$  for $p$ odd. Note that  $H'$ is an exact copy of $H$. Since, as we have already noted in Proposition~\ref{eksikey}, 
the sign assignment on the regions of convex $T$ determines a priori distinct real tight structures on $M$, the numbers above for the possible tight structures must be multiplied by 2.

Now let us  discuss $\# \Gamma$. 
Consider a Legendrian core $K$ of $H$ and its equivariant copy $K'$ in $H'$. A sufficiently small standard neighborhood $\nu(K)\subset H$ of $K$ with convex boundary has  $\#\Gamma=2$  (same situation for the equivariant copy).
Suppose there is a bypass disk $D$ on $\partial\nu(K)$. Let 
$D\cap \mathrm{fix}(c)\neq\emptyset$. Then $D$ can be deformed so that the intersection contains a single point and further it can be made smaller such that it avoids the real point. 
Finally, if   still
$D\cap T\neq\emptyset$ then $T$ can be pushed further (and $H$, $H'$ and the real structure are modified accordingly) so that the intersection becomes empty.  
Now we can peel off a basic slice in $M-\nu(K)$ staying in $H$. Equivariantly we peel off a basic slice in $M-\nu(K_2)$ staying in $H'$ too.
Thus, as in \cite{ho}, we guarantee that the {\em inner} boundary of the closure of $H-(\nu(K)\cup \mbox{ (basic slices)})$ (and of the equivariant copy) is convex with $\#\Gamma=2$. 
In this manner  peeling of basic slices equivariantly we obtain the remaining part  $P=M-(\nu(K)\cup\nu(K')\cup\; \mathrm{ (all\;basic\;slices)}$, which is a $c$-real tight manifold with convex boundary. We may suppose that  $P \cong T^2\times [0,2]$; 
$c(T_1)=T_1$, reversing the orientation;
$c(T_t)=T_{2-t}$, reversing the orientation;
$T_0$ and $T_2$ are convex with $\#\Gamma_j=2$. 
Moreover $s_2=s_0$; in fact
it is obvious that $dist(s_2,s_0)\leq 1$. If it is 1, there is a bypass disk on $T_0$ not touching $T$ as above and a symmetric bypass disk on $T_2$. Then the surfaces $T'_0$ and $T'_2$ obtained after bypass attachment have slopes $s_2$ and $s_0$ respectively. However $T'_2$ is between $T'_0$ 
and  $T_2$, both of which have slope $s_2$. This contradicts with the fact that all the thick tori in
this setup are minimally twisting \cite{ho}.

Now for a contradiction let us assume  $\#\Gamma=\#\Gamma_{T_1}=2u>2$. 
In the words of \cite{ho}, this $T^2\times [0,1]$ is a thick torus increasing the torus division number $\#\Gamma/2$.
Of course, this thick torus has its equivariant copy $T^2\times [1,2]$ in $H'$.

We will observe that such an equivariant tight thick torus $T=(T^2\times [0,2],c,\xi)$ 
that alters the torus division number may not exist. 
Let the fixed point set be parallel to $f_1=(a,b)$ on $T_1$; thus $c(f_1)=f_1$.
Furthermore we assume that the slope of the characteristic foliation on every $T_k$ is $b/a$.
We denote by $A$ the equivariant convex annulus bounded by $f_0\cup f_2$, containing $f_1$. 
Now, since $s_0=s_1$,
the cardinality
$|\Gamma_{T_j}\cap A|$ are the same for $j=0,2$ and 2 more for $j=1$. 
Since $\xi$ is nonrotative, 
$\Gamma_A$ must have at least two dividing arcs which go across from $T_0$ to $T_1$ and therefore 
from $T_2$ to $T_1$ (see \cite{ho}). Then there must be a dividing arc $\gamma$ on $A$ parallel to $A\cap T_1$.
the curve $\gamma$ and its symmetric copy $c(\gamma)$ bound an overtwisted disk on $A$ together, which
contradicts the fact that $\xi$ is tight (see Figure~\ref{incrno1}). \hfill $\Box$
\begin{figure}[h]
\begin{center}
\resizebox{5cm}{!}
{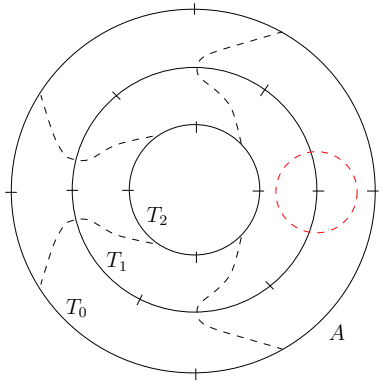}
\caption{The proof that $\#\Gamma=2$ on the equivariant convex Heegaard torus.}
\label{incrno1}
\end{center}
\end{figure}

\subsection{Twisting numbers.} 
\label{buk}
One way to distinguish the real tight structures (up to equivariant contactomophisms) is to calculate the rational twisting number $tb_{\Q}$ of the Legendrian real parts. Here we compute that number explicitly for the real contact  structures for types~$B$, $C$ and $C'$ in all existing cases.

(type~$B$) Consider the proof of Lemma~\ref{typeB} and let $q=p-1 $. When $p$ is even, the fixed point set is the union of the cores $L_1$ and $L_2$ of $H_1$ and $H_2$ respectively. Otherwise fixed point set is $L_1$.  For each real knot in each case, one can compute that $tb_{\Q}=-1/p$. 
Recall that the slopes $s_1$ and $s_2$ are $-1$. Now, for instance, when $p$ is even, $p L_1$ bounds a topological disk $D$ 
since $ (1-p,p)\mapsto (1,0)$ under the gluing map. Hence the rational framing on $L_1$ provided by $D$ is $p(p-1)/p^2=(p-1)/p$.
Therefore  $tb_{\Q}(L_2)=(p-1)/p - 1 = -1/p$. 
Similar computation in case $p$ odd.

(type~$C'$) In this case we set off by assuming the existence of a real contact Heegaard decomposition of genus 1. The gluing map of the Heegaard decomposition is $\Phi$.
We observe that $\Phi((-q,p))=(1,0)$ so that $(-q,p)$ bounds an oriented 
horizontal disk $D_2$ in $H_2$. Set $g=gcd(q-1,p)$ and $h=gcd(q+1,p)$.
The fixed point set of the real structure is one or two copies of $K=\frac{1}{g}(q-1,-p)$ 
on $T=\partial H_1$ depending on $p$ being  odd or even respectively.
Moreover the equivariant dividing curve $\Gamma$
on $H_1$ is parallel to $\frac{1}{h}(q+1,-p)$ and $\#\Gamma=2$. 
Then we have $gK=\partial (D_2\cup D_1)$, with $D_1$ being the meridional disk in $H_1$. 
Let $S$ denote the surface obtained from $D_2\cup D_1$ by adding a band at each crossing of  $D_2\cap D_1$ connecting $D_2$ to $D_1$ so that the orientations are respected. We see that there are $(1,0)\cdot (-q,p)=p$ bands in $S$, which
is minus the difference $(T,S)$ between framings coming from $T$ and $S$ along $gK$ as well.
Now, the twisting of the contact framing along $K$ with respect to $T$
is $(K,T)= \frac{1}{2}(K\cdot \Gamma) = -2p/gh$. 
As a result we have
$$tb_{\Q}(K)=tb_{\Q}(K,S)= \frac{1}{g^2} (T,S) - (K,T) =p/g^2 - 2p/gh.$$

In particular when $q=p-1$ and $p$ is odd we have $h=p$, $g=1$  so that $tb_{\Q}(K)= p-2$.
For the case $q=p-1$ and $p$ is even we have $h=p$ and $g=2$ so that $tb_{\Q}(K)= p/4 - 1$.

(type~$C$) We assume the existence of a real contact Heegaard decomposition of genus 1. The gluing map is  $-\Phi$ this time.
Then $(q,-p)$ bounds $D_2$;
$K=\frac{1}{h}(q+1,-p)$ and 
$\Gamma$  is parallel to $\frac{1}{g}(q-1,-p)$.
Then we have $hK=\partial (D_2\cup D_1)$. As above $S$ is obtained from $D_2\cup D_1$ by adding a band at each crossing respecting the orientations
and there are $p$ bands again.
The twisting of the contact framing along $K$ with respect to $T$
remains the same: $(K,T)= -K\cdot \Gamma = -2p/gh$.  As a result we have
$$tb_{\Q}(K)= \frac{1}{h^2} (T,S) - (K,T) =p/h^2 - 2p/gh.$$

In particular when $q=1$ and $p$ is odd we have $tb_{\Q}(K)=    p-2$.
For the case $q=1$ and $p$ is even we have $tb_{\Q}(K)= p/4-1$.

In case $q=p-1$ and $p$ is odd we have $tb_{\Q}(K)= 1/p - 2$.
For the case $q=1$ and $p$ is even we have $tb_{\Q}(K)= 1/p - 1$. \\

There is a unique real tight $L(p,q)$ of type~$B$ for $q=1,p-1$. If $l^*_{C}(p,p-1)$ were positive (or if $l^*_{C'}(p,1)$ were positive) then the twisting computation would agree. Since they do not, except for $p=2$, we conclude 

\begin{theorem}
Suppose $p>2$, $q=1,p-1$. For a $B$-real tight $L(p,q)$, a real contact Heegaard decomposition must have genus greater than $1$; 
i.e. $l^*_{C'}(p,1)=0$ and $l^*_{C}(p,p-1)=0$. Equivalently, an invariant Heegaard torus in a $B$-real tight $L(p,q)$ cannot be made convex equivariantly.
\label{l*Cp-1}
\end{theorem}

In the case of type~$A$ 
we content ourselves with the following result:

\begin{prop} There is no real contact Heegaard decomposition of genus $1$ for a $A$-real tight $L(p,1)$ induced from a real open book decomposition. 
\end{prop}

\begin{proof} Suppose that there is such. Then 
the real open book has annular pages and its monodromy is ${t_a}^{\pm p}$, where $t_a$ is a right handed Dehn twist about the core circle $a$ in the annulus. The real open book supports a real $L(p, 1)$ if and only if the monodromy is ${t_a}^{- p}$. However in this case the contact structure is overtwisted.
\end{proof}

\section{Construction for $q=p-1$ via surface singularities}
\label{q=p-1}

Here we show that all 
the real tight lens spaces given by surgery diagrams for $L(p,p-1)$ in the first row of Figure~\ref{tumtablo} are real contact link manifolds of isolated real algebraic surface singularities. Meanwhile note that all of the 
surgery diagrams for $L(p,1)$ in the second row of Figure~\ref{tumtablo} 
is virtually overtwisted except two of them. Since universally tightness is 
necessary for Milnor fillability  \cite{lo}, these virtually overtwisted ones are not link manifolds of isolated real algebraic surface singularities.

We consider the family of surface singularities
$A^{\pm}_{p-1}=\{\pm x^{p} - y^2 +  z^2=0\}\subset \C^3$ and
the link manifolds $X^{\pm}=A^{\pm}_{p-1}\cap S^5_{\epsilon}$. We supply these manifolds  with the complex conjugation  as the real structure and consider the induced real tight structure $c^{\pm}$ on $X^{\pm}$.
Then $X^{\pm}$ is a real tight lens space $L(p,p-1)$.
It is straightforward to see that  for each $p$ these two contact manifolds are ambiently contact isotopic and as real manifolds they are equivariantly diffeomorphic. 
However these manifolds turn out to be distinct up to equivariant contact {\em isotopy}.

Using the techniques in \cite{oz} we will distinguish between these real tight lens spaces. 
In particular we will match $tb_{\Q}$ of  the real parts of $X^{\alpha,\beta,\gamma}$ with those computed in Section~\ref{buk}.

\begin{figure}[b]
\begin{center}
\includegraphics[width=6cm]{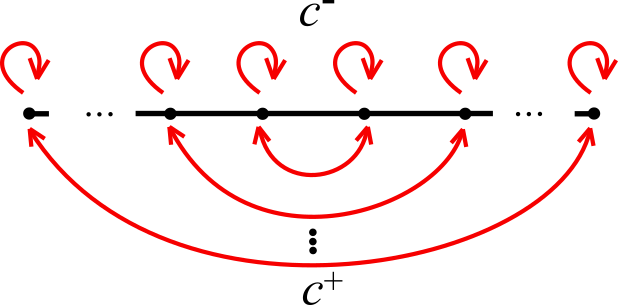}
\caption{The two real structures on the link of $A_{p-1}$, $p$ odd.}
\label{ApGraf}
\end{center}
\end{figure}

We consider the minimal resolution (dual) graph $G$ of $A^{\pm}_{p-1}$; it is a linear graph of 
length $(p-1)$ with  $(-2)$-vertices (see Figure~\ref{ApGraf}).
Any involution on $G$ describes a real structure on the plumbed 4-manifold hence on the lens space;
conversely any real structure on $A^{\pm}_{p-1}$ 
induces an involution on $G$. One way to construct $G$ is to consider the resolution of 
$A^{\pm}_{p-1}$ as a double covering branched along a resolved curve.
Inspecting the positive and negative sides of the real part of the resolved curve in appropriate charts, as 
in \cite[Proposition~1]{oz}, we deduce that the induced involution of the real structure $c^{-}$ on $G$ 
fixes all vertices in $G$ and that induced by $c^+$ acts as a symmetry with respect to the midpoint (see Figure~\ref{ApGraf})

Having said these, let $p=2k+1$.
The fixed point set $S^{\pm}=fix(c^{\pm})$ in the resolved complex surface is a oriented real surface.  
The real part  $L^{\pm}=X^{\pm}\cap S^{\pm}$ of $X^{\pm}$ is a connected Legendrian knot. The genera of $S^-$ and $S^+$ are $(p-1)/2$ and 0 respectively. Now the calculation for $L^-$ is easy:
$$ tb_{\Q}(L^-)=-\chi(S^-)=p-1-2+1=p-2.$$

\begin{figure}[t]
\begin{center}
\includegraphics[width=5cm]{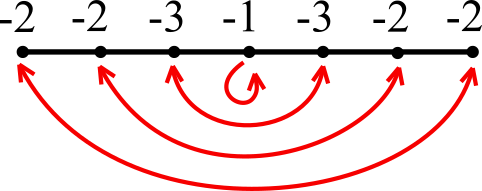}
\caption{Blown up  $A_6$.}
\label{blow}
\end{center}
\end{figure}
For $L^+$, we blow up the resolved complex surface once more.
One has to look more closely to the point of blowing-up and the equivariant setting. 
We blow up at the point of intersection of the equators of the equivariant pair of  exceptional divisors in the very middle. The complex conjugation maps those equators onto each other without fixing any point. 
Thus the conjugation extends over the new exceptional divisor as a $\pi$-rotation 
with respect to the equatorial plane so that its equator (which intersects each of the previous equators at a single point) is equivariant without  any fixed point. So we obtain  a real variety $\tilde{A}$ with the resolution graph as in Figure~\ref{blow}.  
Now the real surface $\tilde{S}^+\subset \tilde{A}$ 
is a punctured $\R P^2$ with boundary $L^+$. The real circle $E$ on the $(-1)$-sphere 
is 1-sided in $\tilde{S}^+$ in and bounds a disk in $X^+$. The contribution of this disk to
$tb_{\Q}(L^+)$ is easy to determine using \cite[Proposition~9 and Theorem~1]{oz}. Thus we get
\begin{align*}
tb_{\Q}(L^+) &=-\chi(S^+ - E) + (-1+2 \cdot [-3,\stackrel{k}{\widehat{-2,\ldots,-2}}]^{-1}) \\
& = 0 -1 + 2(-3-\frac{k-1}{-k})^{-1} =   -\frac{1}{2k+1} = -1/p.
\end{align*}

\begin{figure}[b]
\begin{center}
\includegraphics[width=7cm]{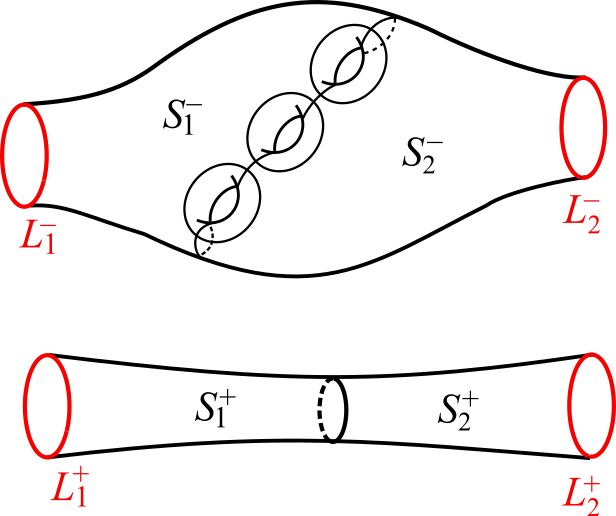}
\caption{Real parts $L_i^{\pm}$ bound the real surfaces in $\tilde{A}_{p-1}$. For $p=2k$, the former has genus $k$ and the latter has genus $0$.}
\label{ApTB}
\end{center}
\end{figure}

Now let $p=2k$. 
The real structure $c^{-}$ fixes all vertices in  the minimal resolution graph; $c^{+}$ is the mirror symmetry with respect to the middle vertex. The real part of $X^{\pm}$ consists of two circles. Let us denote an arbitrarily chosen one by $L^{\pm}$ (the calculation below applies to both components).
The genus of $S^-$ is $k-1$ while $S^+$ is an annulus. Following the idea in \cite[Section~5.1]{oz}
and using Figure~\ref{ApTB}  we can compute the rational twistings  as follows
$$ tb_{\Q}(L^-)=-\chi(S^-_1)-k\cdot \frac{1}{2}=k-1-\frac{k}{2}=\frac{k}{2}-1=\frac{p}{4}-1$$
and $$tb_{\Q}(L^+)=-\chi(S^+_1) + \frac{1}{4}(- 2+2\cdot\frac{k-1}{k})=0- \frac{1}{2k}=-\frac{1}{p}.$$ 

We note that each rational Thurston-Bennequin number that we have obtained above is maximal for the corresponding knot type. For this we refer, for example, \cite[Theorem~1.3]{ch}.

\begin{remark}
\label{blow2}
We would have obtained another real tight $L(p,p-1)$ if we had extended the real structure over the middle exceptional sphere in a non-canonical manner, namely as a $\pi$-rotation in a way that its equator has exactly two fixed points.
This would give a different equivariant contact surgery diagram.
\end{remark}

Comparing with the surgery diagrams obtained and with the values of $tb_{\Q}$ in Section~\ref{buk} and taking into account this last remark,  we conclude

\begin{theorem} 
\label{eksEn}
The unique real tight lens space of type~$B$ in case $q=p-1$ is exactly the link manifold $A^+_{p-1}$ of the singularity $\{+x^{p} - y^2+z^2=0\}$.
The link manifold $A^-_{p-1}$ of 
the singularity $\{-x^{p} - y^2 + z^2=0\}$ is a real tight $L(p,p-1)$ of type~$A$,
which is depicted in the upper left diagram in Figure~\ref{tumtablo}. The diagram in  Figure~\ref{tabloremark} is obtained from $A^+_{p-1}$ by simply modifying the real structure in the middle 2-handle.
\end{theorem}

Acknowledgment. The research for this paper was started during a \textit{Research in Pairs} stay at IMBM ({\.I}stanbul Center for Mathematical Sciences) in January 2020. We thank IMBM for its support, and for creating an inspiring environment. SO was partially supported
by T\"UB{\.I}TAK grant 119F411. F\"O thanks Nermin Salepci for the fruitful  
discussions that took place during his visit at Institut Camille Jordan, Universit\'e Lyon I.

\end{document}